\documentclass{article}

\usepackage{empheq}
\usepackage{caption}
\usepackage{subcaption}

\usepackage{float}
\floatstyle{ruled}
\newfloat{model}{thp}{lop}
\floatname{model}{Model}

\usepackage{multicol}

\usepackage{xcolor}

\newcommand{\amdqtwo}{Dell PowerEdge R415 servers with Dual 3.1GHz AMD 6-Core Opteron 4334 CPUs and 64GB of memory}

\usepackage{xspace}

\usepackage{url}
\usepackage{amssymb,amsmath,amsthm,bm,mathtools,mathabx}
\usepackage{graphicx}
\usepackage{authblk}
\def\rit{\mathbb{R}}

\usepackage{hyperref} 
\hypersetup{
colorlinks=true, 
breaklinks=true,
urlcolor= blue, 
linkcolor= blue,
citecolor=red, 
}

\newtheorem{theorem}{Theorem}[section]
\newtheorem{lemma}[theorem]{Lemma}
\newtheorem{proposition}[theorem]{Proposition}
\newtheorem{corollary}[theorem]{Corollary}

\usepackage{algorithm2e}

\newcommand{\refacwpf}{(\hyperref[model:ac_opf_w]{AC-OPF-W})}

\begin{document}

\title{Strengthening the SDP Relaxation of AC Power Flows with \\ Convex Envelopes, Bound Tightening, and Lifted Nonlinear Cuts} 
\author[1,2]{Carleton Coffrin}
\author[1,2]{Hassan Hijazi}
\author[3]{Pascal Van Hentenryck}

\affil[1]{Optimisation Research Group, NICTA}
\affil[2]{The Australian National University, Canberra, 2601, Australia}
\affil[3]{University of Michigan, Ann Arbor, MI-48109, USA}
\maketitle

\abstract{ This paper considers state-of-the-art convex relaxations
  for the AC power flow equations and introduces valid cuts based
  on convex envelopes and lifted nonlinear constraints. These valid
  linear inequalities strengthen existing semidefinite and quadratic
  programming relaxations and dominate existing cuts proposed in the
  literature.  Combined with model intersection and bound
  tightening, the new linear cuts close 8 of the remaining 16 open
  test cases in the NESTA archive for the AC Optimal Power Flow
  problem.}

\section*{Nomenclature}
\begin{multicols}{2} 

\begin{description}
  \item [{$N$}]  - The set of nodes in the network 
  \item [{$E$}]  - The set of {\em from} edges in the network 
  %
  \item [{$\bm i$}] - imaginary number constant
  \item [{$I$}] - AC current
  \item [{$S = p+ \bm iq$}] - AC power
  \item [{$V = v \angle \theta$}]  - AC voltage
  \item [{$Z = r+ \bm ix$}] - Line impedance
  \item [{$Y = g + \bm ib$}]  - Line admittance
  \item [{$W = w^R+ \bm i w^I $}]  - Product of two AC voltages
  %
  \item [{$s^u$}] - Line apparent power thermal limit
  \item [{$\theta_{ij}$}] - Phase angle difference (i.e. $\theta_i - \theta_j$)
  \item [{$\phi$}] - Phase angle difference center
  \item [{$\delta$}] - Phase angle difference offset
  \item [{$S^d = p^d+ \bm iq^d$}] - AC power demand
  \item [{$S^g = p^g+ \bm iq^g$}] - AC power generation
  \item [{$c_0,c_1,c_2$}] - Generation cost coefficients 
 %
   \item [{$\Re(\cdot)$}] - Real component of a complex number
   \item [{$\Im(\cdot)$}] - Imaginary component of a complex number
   \item [{$(\cdot)^*$}] - Conjugate of a complex number
   \item [{$|\cdot|$}] - Magnitude of a complex number, $l^2$-norm
  %
  %
  \item [{$x^u$}] - Upper bound of $x$
  \item [{$x^l$}] - Lower bound of $x$
  \item [{$x^\sigma$}] - Sum of the bounds (i.e. $x^l + x^u$)
  \item [{$\widecheck{x}$}] - Convex envelope of $x$
  \item [{$\bm x$}] - A constant value
\end{description}

\end{multicols}

\clearpage

\section{Introduction}

Convex relaxations of the AC power flow equations have attracted
significant interest in recent years. These include the semidefinite
Programming (SDP) \cite{Bai2008383}, Second-Order Cone (SOC)
\cite{Jabr06}, Convex-DistFlow (CDF) \cite{6102366}, and the recent
Quadratic Convex (QC) \cite{QCarchive} and Moment-Based
\cite{7038397,6980142} relaxations. Much of the excitement underlying
this line of research comes from the fact that the SDP relaxation was
shown to be tight \cite{5971792} on a variety of AC Optimal Power Flow
(AC-OPF) test cases distributed with Matpower \cite{matpower}, opening
a new avenue for accurate, reliable, and efficient solutions to a
variety of power system applications. Indeed, industrial-strength
optimization tools (e.g., Gurobi \cite{gurobi}, Cplex \cite{cplex},
Mosek \cite{mosek}) are now readily available to solve various classes
of convex optimization problems.

It was long thought that the SDP relaxation was the tightest convex
relaxation of the power flow equations.  However, recent works have
demonstrated that realistic test cases can exhibit a non-zero
optimality gap with this relaxation \cite{nesta,7056568}.  These new
test cases also demonstrate that the QC relaxation can be tighter than
the SDP relaxation in some cases \cite{qc_opf_tps}.  This result was
further extended in \cite{cp_qc_fp} to show that the QC relaxation,
when combined with a bound tightening procedure, is stronger than the
SDP relaxation in the vast majority of cases. However, at least 16
AC-OPF test cases in NESTA v0.6.0 \cite{nesta} still exhibit an
optimality gap above 1\% using the relaxation
developed in \cite{cp_qc_fp},

This paper builds on these results (i.e., \cite{QCarchive, qc_opf_tps,
  cp_qc_fp,opfBranchDecomp}) trying to
further improve existing convex relaxations in order to close the
optimality gap on the remaining open test cases. Its main
contributions can be summarized as follows. The paper
\begin{enumerate}
\item develops stronger power flow relaxations dominating state-of-the-art methods;

\item proposes a novel approach to generating valid inequalities for non-convex programs;

\item utilizes this novel approach to develop {\em Extreme cuts} and
  {\em lifted nonlinear cuts} for the AC power flow equations, which
  can be used to strengthen power flow relaxations;

\item presents computational results demonstrating that the optimality
  gap on many of the open test cases can be reduced to less than 1\%,
  using a combination of the methods developed herein.
\end{enumerate}
\noindent
The computational study is conducted on 71 AC Optimal Power Flow test
cases from NESTA v0.6.0, which feature realistic side-constraints and
incorporate bus shunts, line charging, and transformers.

The rest of the paper is organized as follows.  Section
\ref{sec:ac:pf} reviews the formulation of the AC-OPF problem from
first principles and presents the key operational side constraints for
AC network operations.  Section \ref{sec:relaxations} derives the
state-of-the-art SDP and QC relaxations.  Section \ref{sec:tighten}
presents three orthogonal and compositions methods for tightening
convex relaxations and applies those to the AC power flow constraints.
Section \ref{sec:experiments} reports the benefits of the various
tightening methods on AC-OPF test cases, and Section
\ref{sec:conclusion} concludes the paper.

\section{AC Optimal Power Flow}
\label{sec:ac:pf}

This section reviews the specification of AC Optimal Power Flow (AC-OPF)
and introduces the notations used in the paper. In the equations,
constants are always in bold face. 

A power network is composed of a variety of components such as buses, lines, generators, and loads.  
The network can be interpreted as a graph $(N,E)$ where the set of buses $N$ represent the nodes and the set of lines $E$ represent the edges. 
Note that $E$ is an undirected set of edges, however each edge $(i,j) \in E$ is assigned a {\em from} side $(i,j)$ and a {\em to} side $(j,i)$, arbitrarily.
These two sides are critically important as power is lost as it flows from one side to another. 
Lastly, to break numerical symmetries in the model and to allow easy comparison of solutions, a reference node $r \in N$ is also specified.

The AC power flow equations are
based on complex quantities for current $I$, voltage $V$, admittance
$Y$, and power $S$, which are linked by the physical properties of
Kirchhoff's Current Law (KCL), i.e.,
\begin{align}
& I^g_i - {\bm I^d_i} = \sum_{\substack{(i,j)\in E}} I_{ij} + \sum_{\substack{(j,i)\in E}} I_{ij} 
\end{align}
Ohm's Law, i.e.,
\begin{align}
& I_{ij} = \bm Y_{ij} (V_i - V_j)  \label{current_flow}
\end{align}
and the definition of AC power, i.e.,
\begin{align}
& S_{ij} = V_{i}I_{ij}^* \label{complex_power}
\end{align}
Combining these three properties yields the AC Power Flow equations, i.e.,
\begin{subequations}
\begin{align}
& S^g_i - {\bm S^d_i} = \sum_{\substack{(i,j)\in E}} S_{ij} + \sum_{\substack{(j,i)\in E}} S_{ij} \;\; \forall i\in N \\ 
& S_{ij} = \bm Y^*_{ij} V_i V^*_i - \bm Y^*_{ij} V_i V^*_j \;\; (i,j),(j,i) \in E
\end{align}
\end{subequations}

\noindent
Observe that $\sum$ over $(i,j)\in E$ collects the edges oriented in the 
{\em from} direction and $\sum$ over $(j,i)\in E$ collects the edges oriented in the {\em to} direction around bus $i\in N$.
These non-convex nonlinear equations define how power flows in the
network and are a core building block in many power system
applications. However, practical applications typically include various
operational side constraints. We now review some 
of the most significant ones.

\paragraph*{Generator Capacities}

AC generators have limitations on the amount of active and reactive
power they can produce $S^g$, which is characterized by a generation
capability curve \cite{9780070359581}.  Such curves typically define
nonlinear convex regions which are most-often approximated by boxes in
AC transmission system test cases, i.e.,
\begin{subequations}
\begin{align}
& \bm {S^{gl}}_i \leq S^g_i \leq \bm {S^{gu}}_i \;\; \forall i \in N 
\end{align}
\end{subequations}

\paragraph*{Line Thermal Limits}

Power lines have thermal limits \cite{9780070359581} to prevent
lines from sagging and automatic protection devices from activating.
These limits are typically given in Volt Amp units and bound 
the apparent power flow on a given line, i.e.,
\begin{align}
& |S_{ij}| \leq \bm {s^u}_{ij} \;\; \forall (i,j),(j,i) \in E 
\end{align}

\paragraph*{Bus Voltage Limits}

Voltages in AC power systems should not vary too far (typically $\pm
10\%$) from some nominal base value \cite{9780070359581}.  This is
accomplished by putting bounds on the voltage magnitudes, i.e.,
\begin{align}
& \bm {v^l}_i \leq |V_i| \leq \bm {v^u}_i \;\; \forall i \in N
\end{align}
A variety of power flow formulations only have variables for the
square of the voltage magnitude, i.e., $|V_i|^2$.  In such cases, the
voltage bound constrains can be incorporated via the following constraints:
\begin{align}
& ( \bm {v^l}_{i} )^2 \leq |V_i|^2 \leq ( \bm {v^u}_{i} )^2 \;\; \forall i \in N \label{eq:v_mag_sqr}
\end{align}

\paragraph*{Phase Angle Differences}

Small phase angle differences are also a design imperative in AC power
systems \cite{9780070359581} and it has been suggested that phase
angle differences are typically less than $10$ degrees in practice
\cite{Purchala:2005gt}. These constraints have not typically been
incorporated in AC transmission test cases \cite{matpower}. However,
recent work \cite{LPAC_ijoc,QCarchive,cp_qc_fp} have observed that
incorporating Phase Angle Difference (PAD) constraints, i.e.,
\begin{align}
&  \bm {\theta^l}_{ij} \leq \angle \! \left( V_i V^*_j \right) \leq \bm {\theta^u}_{ij} \;\; \forall (i,j) \in E \label{eq:pad_1}
\end{align}
is useful in characterizing the feasible space of the AC power flow equations.
This work assumes that the phase angle difference bounds and within the range $(- \bm \pi/2, \bm \pi/2 )$, i.e.,
\begin{align}
& -\frac{\bm \pi}{2} \leq \bm {\theta^l}_{ij} \leq \bm {\theta^u}_{ij} \leq \frac{\bm \pi}{2} \;\; \forall(i,j) \in E
\end{align}
Given the design imperatives of AC power systems \cite{9780070359581,Purchala:2005gt}, this does not appear to be a significant limitation.
Observe also that these PAD constraints \eqref{eq:pad_1} can be
implemented as a linear relation of the real and imaginary components
of $V_iV^*_j$ \cite{6810520},
\begin{align}
& \tan(\bm {\theta^l}_{ij}) \Re\left(V_iV^*_j\right) \! \leq \!  \Im\left(V_iV^*_j\right) \! \leq \! \tan(\bm {\theta^u}_{ij}) \Re\left(V_iV^*_j\right) \;\; \forall(i,j) \in E \label{eq:w_pad}
\end{align}
The usefulness of this formulation will be apparent later in the
paper.

\paragraph*{Other Constraints}
Other line flow constraints have been proposed, such as, active power
limits and voltage difference limits \cite{5971792,6810520}.  However,
we do not consider them here since, to the best of our knowledge, test
cases incorporating these constraints are not readily available.

\paragraph*{Objective Functions}

The last component in formulating AC-OPF problems is an objective
function. The two classic objective functions are line loss
minimization, i.e.,
\begin{align}
& \mbox{minimize: } \sum_{i \in N} \Re(S^g_i)  \label{eq:loss_min}
\end{align}
and generator fuel cost minimization, i.e.,
\begin{align}
& \mbox{minimize: } \sum_{i \in N} \bm c_{2i} (\Re(S^g_i))^2 + \bm c_{1i}\Re(S^g_i) + \bm c_{0i} \label{eq:fule_min}
\end{align}
Observe that objective \eqref{eq:loss_min} is a special case of
objective \eqref{eq:fule_min} where $\bm c_{2i}\!=\!0, \bm
c_{1i}\!=\!1, \bm c_{0i}\!=\!0 \;\; (i \! \in \! N)$ \cite{6153415}.
Hence, the rest of this paper focuses on objective
\eqref{eq:fule_min}.

\paragraph*{The AC Optimal Power Flow Problem}

Combining the AC power flow equations, the side constraints, and the
objective function, yields the well-known AC-OPF formulation presented
in Model \ref{model:ac_opf_w}.  This formulation utilizes a voltage product factorization (i.e. $V_i V_j^* = W_{ij} \;\; \forall (i,j)\in E$), a complete derivation of this formulation can be found in \cite{qc_opf_tps}.
In practice, this non-convex nonlinear optimization problem is typically solved with numerical methods \cite{744492,744495}, which provide locally optimal solutions if they converge to a feasible point.


\begin{model}[t]
\caption{ The AC Optimal Power Flow Problem with the W Factorization (AC-OPF-W).}
\label{model:ac_opf_w}
\begin{subequations}
\vspace{-0.2cm}
\begin{align}
\mbox{\bf variables: } \nonumber \\
& S^g_i \in ( \bm {S^{gl}}_i, \bm {S^{gu}}_i) \;\; \forall i\in N \nonumber \\
& V_i \in ( \bm {V^l}_i, \bm {V^u}_i ) \;\; \forall i\in N \nonumber \\
& W_{ij} \in ( \bm {W^l}_{ij}, \bm {W^u}_{ij} )  \;\; \forall i \in N,  \forall j \in N \\
& S_{ij} \in (\bm {S^{l}}_{ij},\bm {S^{u}}_{ij})\;\; \forall (i,j),(j,i) \in E \nonumber \\
\mbox{\bf minimize: }  \\
& \sum_{i \in N} \bm c_{2i} (\Re(S^g_i))^2 + \bm c_{1i}\Re(S^g_i) + \bm c_{0i} \label{w_obj} \\
\mbox{\bf subject to: } \nonumber \\
& \angle V_{\bm r} = 0 \\
&  W_{ij} = V_iV_j^* \;\; \forall (i,j)\in E \label{w_2} \\
& S^g_i - {\bm S^d_i} = \sum_{\substack{(i,j)\in E}} S_{ij} + \sum_{\substack{(j,i)\in E}} S_{ij}\;\; \forall i\in N \label{w_5} \\ 
& S_{ij} = \bm Y^*_{ij} W_{ii} - \bm Y^*_{ij} W_{ij} \;\; \forall (i,j)\in E \label{w_6} \\
& S_{ji} = \bm Y^*_{ij} W_{jj} - \bm Y^*_{ij} W_{ij}^* \;\; \forall (i,j)\in E \label{w_7} \\
& |S_{ij}| \leq (\bm {s^u}_{ij}) \;\; \forall (i,j),(j,i) \in E \label{w_8} \\
& \tan(\bm {\theta^l}_{ij}) \Re(W_{ij}) \leq \Im(W_{ij}) \leq \tan(\bm {\theta^u}_{ij}) \Re(W_{ij}) \;\; \forall (i,j) \in E \label{w_9}
\end{align}
\end{subequations}
\end{model}

A key message throughout this work and related works \cite{qc_opf_tps,cp_qc_fp} is that the bounds on the decision variables are a critical consideration in the AC-OPF problem.
Hence, the variable bounds are explicitly specified in Model \ref{model:ac_opf_w}.
Noting that bounds on the variables $V, W, S$ are most often omitted from power network datasets, we precent valid bounds here.
Suitable bounds for $V$ and $S$ can be deduced from the bus voltage and thermal limit constraints as follows,
\begin{subequations}
\begin{align}
& \bm {V^{u}}_i = \bm {v^u}_{i} + \bm i \bm {v^u}_{i}, \bm {V^{l}}_{ij} = -(\bm {v^u}_{i} + \bm i \bm {v^u}_{i})  \;\; \forall i \in N \nonumber \\
& \bm {S^{u}}_{ij} = \bm {s^u}_{ij} + \bm i \bm {s^u}_{ij}, \bm {S^{l}}_{ij} = -(\bm {s^u}_{ij} + \bm i \bm {s^u}_{ij}) \;\; \forall(i,j) \in E \nonumber
\end{align}
\end{subequations}
A derivation of these bounds can be found in \cite{nfcp_report}.  The bounds on the diagonal of the $W$ are as follows, 
\begin{align}
& \bm {W^{u}}_{ii} = \bm (\bm {v^u}_{i})^2 + \bm i 0, \bm {W^{l}}_{ii} = (\bm {v^l}_{i})^2 + \bm i 0 \;\; \forall i \in N \nonumber
\end{align}
These come directly from the bus voltage constraints \eqref{eq:v_mag_sqr}.

The off-diagonal entries of $W$ are broken into two groups, those belonging to $E$ and those not belonging to $E$.

\begin{lemma}
\label{lemma:w_non_egde_vars_bounds}
$\bm {W^{u}}_{ij} = \bm {v^u}_{i} \bm {v^u}_{j} + \bm i \bm {v^u}_{i} \bm {v^u}_{j}, \bm {W^{l}}_{ij} = -\bm {v^u}_{i} \bm {v^u}_{j} - \bm i \bm {v^u}_{i} \bm {v^u}_{j} \;\; \forall (i,j) \not\in E$ 
are valid bounds in \refacwpf.
\end{lemma}
\begin{proof}
Recall that the one of the real number representations of $W_{ij}$ is,
\begin{align}
W_{ij} = v_{i}v_{j}\cos(\theta_{ij}) + \bm i v_{i}v_{j}\sin(\theta_{ij})
\end{align}
Observe that $v_{i} \geq 0$,$v_{j} \geq 0$ and that no bounds are imposed on $\theta_{ij}$ between the buses not in $E$.  
Hence, the domains of both trigonometric functions are $(-1,1)$.
Consequently, the magnitude of each expression can be no greater than $\bm {v^u}_{i} \bm {v^u}_{j}$ and the feasible interval is $(- \bm {v^u}_{i} \bm {v^u}_{j}, \bm {v^u}_{i} \bm {v^u}_{j})$ in both cases.
\end{proof}

\begin{lemma}
\label{lemma:w_egde_vars_bounds}
\begin{align}
\bm {W^{u}}_{ij} &= 
\begin{cases*}
\bm {v^u}_{i} \bm {v^u}_{j} \cos(\bm {\theta^u}_{ij}) + \bm i \bm {v^l}_{i} \bm {v^l}_{j} \sin(\bm {\theta^u}_{ij}) & \text{if  $\bm {\theta^l}_{ij}, \bm {\theta^u}_{ij} \leq 0$ } \\
\bm {v^u}_{i} \bm {v^u}_{j} \cos(\bm {\theta^l}_{ij}) + \bm i \bm {v^u}_{i} \bm {v^u}_{j} \sin(\bm {\theta^u}_{ij})  & \text{if  $\bm {\theta^l}_{ij}, \bm {\theta^u}_{ij} \geq 0$ } \\
\bm {v^u}_{i} \bm {v^u}_{j} + \bm i \bm {v^u}_{i} \bm {v^u}_{j} \sin(\bm {\theta^u}_{ij})  & \text{if  $\bm {\theta^l}_{ij} < 0, \bm {\theta^u}_{ij} > 0$ }
\end{cases*} \;\; \forall (i,j) \in E  \nonumber \\
\bm {W^{l}}_{ij} &= 
\begin{cases*}
\bm {v^l}_{i} \bm {v^l}_{j} \cos(\bm {\theta^l}_{ij}) + \bm i \bm {v^u}_{i} \bm {v^u}_{j} \sin(\bm {\theta^l}_{ij}) & \text{if  $\bm {\theta^l}_{ij}, \bm {\theta^u}_{ij} \leq 0$ } \\
\bm {v^l}_{i} \bm {v^l}_{j} \cos(\bm {\theta^u}_{ij}) + \bm i \bm {v^l}_{i} \bm {v^l}_{j} \sin(\bm {\theta^l}_{ij})  & \text{if  $\bm {\theta^l}_{ij}, \bm {\theta^u}_{ij} \geq 0$ } \\
\min(\bm {v^l}_{i} \bm {v^l}_{j} \cos(\bm {\theta^l}_{ij}), \bm {v^l}_{i} \bm {v^l}_{j} \cos(\bm {\theta^u}_{ij})) + \bm i \bm {v^u}_{i} \bm {v^u}_{j} \sin(\bm {\theta^l}_{ij})  & \text{if  $\bm {\theta^l}_{ij} < 0, \bm {\theta^u}_{ij} > 0$ }
\end{cases*} \;\; \forall (i,j) \in E \nonumber
\end{align}
are valid bounds in \refacwpf.
\end{lemma}
\noindent
A proof can be found in Appendix \ref{sec:w_bounds}.

\begin{corollary}
All of the decision variables in Model \ref{model:ac_opf_w} have well defined bounds parameterized by $\bm {v^l}_i, \bm {v^u}_i \; \forall i \in N$ and $\bm {s^u}_{ij}, \bm {\theta^l}_{ij}, \bm {\theta^u}_{ij} \; \forall (i,j) \in E$, which are readily available in power network datasets.
\end{corollary}

\paragraph*{Model Extensions}

In the interest of clarity, AC Power Flows, and their relaxations, are most often presented on the simplest version of the AC power flow equations.  
However, transmission system test cases include additional parameters such as bus shunts, line charging, and transformers, which complicate the AC power flow equations significantly.
In this paper, all of the results focus exclusively on the voltage product constraint \eqref{w_2}.  As a consequence, the results can be seamlessly extended to these more general cases easily by modifying the constant parameters in constraints \eqref{w_5}--\eqref{w_7}.  Real-world deployment of AC-OPF methods require even more extensions, discussed at length in \cite{Capitanescu20111731,real_opf}.  For similar reasons, it is likely that the results presented here will also extend to those real-world variants.

\section{Convex Relaxations of Optimal Power Flow}
\label{sec:relaxations}

Since the AC-OPF problem is NP-Hard \cite{verma2009power,ACSTAR2015} 
and numerical methods provide limited guarantees for determining feasibility and global optimally,
significant attention has been devoted to finding convex relaxations
of Model \ref{model:ac_opf_w}.  Such relaxations are appealing because
they are computationally efficient and may be used to:
\begin{enumerate}
\item bound the quality of AC-OPF solutions produced by locally optimal methods;
\item prove that a particular instance has no solution; 
\item produce a solution that is feasible in the original non-convex
  problem \cite{5971792}, thus solving the AC-OPF and guaranteeing
  that the solution is globally optimal.
\end{enumerate}
The ability to provide bounds is particularly important for the
numerous mixed-integer nonlinear optimization problems that arise in power
system applications.
For these reasons, a variety of convex relaxations of the AC-OPF have
been developed including, the SDP \cite{Bai2008383}, QC
\cite{QCarchive}, SOC \cite{Jabr06}, and Convex-DistFlow
\cite{6102366,distflow_report}. 
Moreover, since the SOC and Convex-DistFlow relaxations have
been shown to be equivalent \cite{6483453,distflow_report} and that the SOC 
relaxation is dominated by the SDP and QC relaxations \cite{qc_opf_tps}, 
this paper focuses on the SDP and QC relaxations and shows how they 
are derived from Model \ref{model:ac_opf_w}.  The key insight is that each relaxation
presents a different approach to convexifing constraints \eqref{w_2},
which are the only source of non-convexity in Model \ref{model:ac_opf_w}.

\paragraph*{The semidefinite Programming  (SDP) Relaxation} 

exploits the fact that the $W$ variables are defined by $V(V^*)^T$,
which ensures that $W$ is positive semidefinite (denoted by $W
\succeq 0$) and has rank 1 \cite{Bai2008383, 5971792, 6345272}.  These
conditions are sufficient to enforce constraints \eqref{w_2}
\cite{doi:10.1137/1038003}, i.e.,
\begin{equation}
W_{ij} = V_iV_j^* \; (i,j \in N) \;\; \Leftrightarrow \;\; W \succeq 0 \; \wedge \; \mbox{rank}(W) = 1 \nonumber
\end{equation}
The SDP relaxation \cite{sdpIntro,doi:10.1137/1038003} then drops the
rank constraint to obtain Model \ref{model:ac_opf_w_sdp}.

\begin{model}[t]
\caption{The SDP Relaxation (AC-OPF-W-SDP).}
\label{model:ac_opf_w_sdp}
\begin{subequations}
\vspace{-0.2cm}
\begin{align}
\mbox{\bf variables: } \nonumber \\
& S^g_i \in ( \bm {S^{gl}}_i, \bm {S^{gu}}_i) \;\; \forall i\in N \nonumber \\
& W_{ij} \in ( \bm {W^l}_{ij}, \bm {W^u}_{ij} )  \;\; \forall i \in N,  \forall j \in N \\
& S_{ij} \in (\bm {S^{l}}_{ij},\bm {S^{u}}_{ij})\;\; \forall (i,j),(j,i) \in E \nonumber \\
\mbox{\bf minimize: } & \eqref{w_obj} \nonumber \\
\mbox{\bf subject to: } & \mbox{\eqref{w_5}--\eqref{w_9}} \nonumber \\
& W \succeq 0 \label{w_sdp}
\end{align}
\end{subequations}
\end{model}

\paragraph*{The Quadratic Convex (QC) Relaxation}
was introduced to preserve stronger links between the voltage
variables \cite{QCarchive}.  It represents the voltages in polar
form (i.e., $V = v \angle \theta$) and links these real variables
to the $W$ variables, along the lines of \cite{780924,4548149,6661462,RomeroRamos2010562}, using the following equations:
\begin{subequations}
\begin{align}
& W_{ii} = v_{i}^2 \;\; i \in N \label{eq:w_link_1} \\
& \Re(W_{ij}) = v_{i}v_{j}\cos(\theta_i - \theta_j) \;\; \forall(i,j) \in E \label{eq:w_link_2} \\
& \Im(W_{ij}) = v_{i}v_{j}\sin(\theta_i - \theta_j) \;\; \forall(i,j) \in E \label{eq:w_link_3}
\end{align}
\end{subequations}
The QC relaxation then relaxes these equations by taking tight convex
envelopes of their nonlinear terms, exploiting the operational limits
for $v_i, v_j, \theta_i - \theta_j$. The convex envelopes for the
square and product of variables are well-known \cite{MacC76}, i.e.,
\begin{equation}
\tag{T-CONV}
\langle x^2 \rangle^T \equiv
\begin{cases*}
\widecheck{x}  \geq  x^2\\
\widecheck{x}  \leq  ( \bm {x^u} + \bm {x^l})x - \bm {x^u} \bm {x^l}
\end{cases*}
\end{equation}
\begin{equation*}
\tag{M-CONV}
\langle xy \rangle^M \equiv
\begin{cases*}
\widecheck{xy}  \geq  \bm {x^l}y + \bm {y^l}x - \bm {x^l}\bm {y^l}\\
\widecheck{xy}  \geq  \bm {x^u}y + \bm {y^u}x - \bm {x^u}\bm {y^u}\\
\widecheck{xy}  \leq  \bm {x^l}y + \bm {y^u}x - \bm {x^l}\bm {y^u}\\
\widecheck{xy}  \leq  \bm {x^u}y + \bm {y^l}x - \bm {x^u}\bm {y^l}
\end{cases*}
\end{equation*}
%
%
%
Under the assumption that the phase angle difference bound is
within $- \bm \pi/2 \leq \bm {\theta^l}_{ij} \leq \bm {\theta^u}_{ij} \leq \bm \pi/2$, relaxations for sine and
cosine are given by:
\begin{equation*}
\tag{C-CONV}
\langle \cos(x) \rangle^C \equiv
\begin{cases*}
\widecheck{cx}  \leq 1 - \frac{1-\cos({\bm {x^m}})}{({\bm {x^m}})^2} x^2\\
\widecheck{cx}  \geq \frac{\cos(\bm {x^l}) - \cos(\bm {x^u})}{(\bm {x^l}-\bm {x^u})}(x - \bm {x^l}) + \cos(\bm {x^l})
\end{cases*}
\end{equation*}
\begin{equation*}
\tag{S-CONV}
\langle \sin(x) \rangle^S \equiv
\begin{cases*}
\widecheck{sx} \leq \cos\left(\frac{\bm {x^m}}{2}\right)\left(x -\frac{\bm {x^m}}{2}\right) + \sin\left(\frac{\bm {x^m}}{2}\right) & \hspace{-2.0cm} \\
\widecheck{sx} \geq \cos\left(\frac{\bm {x^m}}{2}\right)\left(x +\frac{\bm {x^m}}{2}\right) - \sin\left(\frac{\bm {x^m}}{2}\right) & \hspace{-2.0cm} \\
%
\widecheck{sx} \geq \frac{\sin(\bm {x^l}) - \sin(\bm {x^u})}{(\bm {x^l}-\bm {x^u})}(x - \bm {x^l}) + \sin(\bm {x^l}) & if $\bm {x^l} \geq 0$ \\
\widecheck{sx} \leq \frac{\sin(\bm {x^l}) - \sin(\bm {x^u})}{(\bm {x^l}-\bm {x^u})}(x - \bm {x^l}) + \sin(\bm {x^l}) & if $ \bm {x^u} \leq 0$ \\
\end{cases*}
\end{equation*}
\noindent
where $\bm {x^m} = \max(|\bm {x^l}|, |\bm {x^u}|)$ \cite{cp_qc_fp}.
In the following, we abuse notation and use $\langle f(\cdot)
\rangle^{C}$ to denote the variable on the left-hand side of the
convex envelope $C$ for function $f(\cdot)$. When such an expression is used
inside an equation, the constraints $\langle f(\cdot) \rangle^{C}$ are
also added to the model.

\begin{model}[t]
\caption{The QC Relaxation (AC-OPF-W-QC).}
\label{model:ac_opf_w_qc}
\begin{subequations}
\begin{align}
\mbox{\bf variables: } \nonumber \\
& S^g_i \in ( \bm {S^{gl}}_i, \bm {S^{gu}}_i) \;\; \forall i\in N \nonumber \\
& W_{ii} \in ( \bm {W^l}_{ii}, \bm {W^u}_{ii} )  \;\; \forall i \in N \nonumber \\
& W_{ij} \in ( \bm {W^l}_{ij}, \bm {W^u}_{ij} )  \;\; \forall (i,j) \in E \nonumber \\
& S_{ij} \in (\bm {S^{l}}_{ij},\bm {S^{u}}_{ij})\;\; \forall (i,j),(j,i) \in E \nonumber \\
& v_i \angle \theta_i \in ( \bm {v^l}_i - \bm i \bm \infty, \bm {v^u}_i  + \bm i \bm \infty )  \;\; \forall i \in N \nonumber \\
%
%
\mbox{\bf minimize: } & \eqref{w_obj} \nonumber \\
\mbox{\bf subject to: } & \mbox{\eqref{w_5}--\eqref{w_9}, \eqref{w_soc}} \nonumber \\
& |W_{ij}|^2 \leq W_{ii}W_{jj} \;\; \forall (i,j)\in E \label{soc_1} \\
& \theta_{\bm r} = 0 \label{qc_0} \\
& W_{ii} = \langle v_i^2 \rangle^T  \;\; i \in N \label{qc_1} \\
&\Re(W_{ij}) = \langle \langle v_i v_j \rangle^M \langle \cos(\theta_i - \theta_j) \rangle^C \rangle^M \;\; \forall(i,j) \in E \label{qc_2} \\
&\Im(W_{ij}) = \langle \langle v_i v_j \rangle^M \langle \sin(\theta_i - \theta_j) \rangle^S \rangle^M  \;\; \forall(i,j) \in E \label{qc_3} 
\end{align}
\end{subequations}
\end{model}

Convex envelopes for equations
\eqref{eq:w_link_1}--\eqref{eq:w_link_3} can be obtained by composing
the convex envelopes of the functions for square, sine, cosine, and
the product of two variables, i.e.,
\begin{subequations}
\begin{align}
& W_{ii} = \langle v_i^2 \rangle^T  \;\; i \in N \\
&\Re(W_{ij}) = \langle \langle v_i v_j \rangle^M \langle \cos(\theta_i - \theta_j) \rangle^C \rangle^M \;\; \forall(i,j) \in E \\
&\Im(W_{ij}) = \langle \langle v_i v_j \rangle^M \langle \sin(\theta_i - \theta_j) \rangle^S \rangle^M  \;\; \forall(i,j) \in E 
\end{align}
\end{subequations}
The QC relaxation also proposes to strengthen these convex envelopes
with a second-order cone constraint from the well known SOC relaxation \cite{Jabr06}. 
This SOC relaxation takes the absolute square of each voltage product constraint in \eqref{w_2}, refactors it, and then relaxes the
equality into an inequality, i.e.,
\begin{subequations}
\begin{align}
& W_{ij} = V_iV^*_j \label{w_soc_1}\\
& W_{ij}W^*_{ij} = V_iV^*_jV^*_iV_j \label{w_soc_2} \\
& |W_{ij}|^2 = W_{ii}W_{jj} \label{w_soc_3}\\
& |W_{ij}|^2 \leq W_{ii}W_{jj} \label{w_soc}
\end{align}
\end{subequations}
Equation \eqref{w_soc} is a rotated second-order cone constraint which
is widely supported by industrial optimization tools.


The complete QC relaxation is presented in Model \ref{model:ac_opf_w_qc}.
A key observation of the QC relaxation is that the convex envelopes are determined by the variable bounds.
Hence, as the bounds become smaller the strength of the relaxation increases \cite{cp_qc_fp,qc_opf_tps}.

\section{Strengthening Convex Relaxations}
\label{sec:tighten}

It has been established that the SDP and QC relaxations have different strengths and weaknesses and one does not dominate the other \cite{cp_qc_fp,qc_opf_tps}.
In this work we develop a hybrid relaxation, which dominates both formulations.
This is accomplished by considering three orthogonal and compositional approaches to strengthening the SDP relaxation:
\begin{enumerate}
\item Model Intersection (e.g. \cite{Liberti04,Ruiz11})
\item Valid Inequalities (e.g. \cite{7056568,strong_soc_report})
\item Bound Tightening (e.g. \cite{cp_qc_fp,7328765})
\end{enumerate}
The rest of this section explains how each of these ideas is utilized to strengthen the SDP relaxation.
%

%
%

\subsection{Model Intersection}

\begin{model}[t!]
\caption{The Combined SDP \& QC Relaxation (AC-OPF-W-SDP+QC).}
\label{model:ac_opf_w_sdp_qc}
\begin{subequations}
\vspace{-0.2cm}
\begin{align}
\mbox{\bf variables: } \nonumber \\
& S^g_i \in ( \bm {S^{gl}}_i, \bm {S^{gu}}_i) \;\; \forall i\in N \nonumber \\
& W_{ij} \in ( \bm {W^l}_{ij}, \bm {W^u}_{ij} )  \;\; \forall i \in N,  \forall j \in N \\
& S_{ij} \in (\bm {S^{l}}_{ij},\bm {S^{u}}_{ij})\;\; \forall (i,j),(j,i) \in E \nonumber \\
& v_i \angle \theta_i \in ( \bm {v^l}_i - \bm i \bm \infty, \bm {v^u}_i  + \bm i \bm \infty )  \;\; \forall i \in N \nonumber \\
\mbox{\bf minimize: } & \eqref{w_obj} \nonumber \\
\mbox{\bf subject to: } & \mbox{\eqref{w_5}--\eqref{w_9}} \nonumber \\
& W \succeq 0 \label{w_sdp_qc_1} \\
& \theta_{\bm r} = 0 \label{w_sdp_qc_2} \\
& W_{ii} = \langle v_i^2 \rangle^T  \;\; i \in N \label{w_sdp_qc_3} \\
&\Re(W_{ij}) = \langle \langle v_i v_j \rangle^M \langle \cos(\theta_i - \theta_j) \rangle^C \rangle^M \;\; \forall(i,j) \in E \label{w_sdp_qc_4} \\
&\Im(W_{ij}) = \langle \langle v_i v_j \rangle^M \langle \sin(\theta_i - \theta_j) \rangle^S \rangle^M  \;\; \forall(i,j) \in E \label{w_sdp_qc_5} 
\end{align}
\end{subequations}
\end{model}

Given that the SDP and QC relaxations have different strengths and weaknesses \cite{qc_opf_tps}, a natural and strait-forward way to make a model that dominates both relaxations is to combine them, yielding a feasible set that is the intersection of both relaxations.  Model \ref{model:ac_opf_w_sdp_qc} presents such a model.  

Observe that the second order cone constraint in the QC \eqref{soc_1} is redundant in Model \ref{model:ac_opf_w_sdp_qc} and can be omitted.
The reasoning is that the positive semidefinite constraint \eqref{w_sdp_qc_1} ensures that every sub-matrix of $W$ is positive semidefinite \cite{Prussing86}. 
This includes the following 2-by-2 sub-matrices for each line,
\begin{align*}
 \begin{bmatrix}
  W_{ii} & W_{ij}\\
  W_{ij}^* & W_{jj} \\
 \end{bmatrix} \succeq 0  \;\; \forall (i,j)\in E 
\end{align*}
%
Applying the determinant characterization for positive semidefinite matrices yields,
\begin{subequations}
\begin{align*}
& 0 \leq W_{ii}W_{jj} - W_{ij}W_{ij}^* \;\; \forall (i,j)\in E \\
& |W_{ij}|^2 \leq W_{ii}W_{jj}  \;\; \forall (i,j)\in E
\end{align*}
\end{subequations}
which is equivalent to \eqref{soc_1}.

\subsection{Valid Inequalities}

It was recently demonstrated how valid inequalities can be used to strengthen the SDP and SOC relaxations of AC power flows \cite{7056568,strong_soc_report}.
In this section we develop three valid inequalities inspired by the fundamental source of non-convexity in the OPF problem,
\begin{align}
& W_{ij} = V_iV^*_j \;\; \forall (i,j) \in E \label{w_v_eq} 
\end{align}
We begin by observing that the non-convex constraint,
\begin{align}
& |W_{ij}|^2 = W_{ii}W_{jj}  \;\; \forall (i,j) \in E \label{w_eq} 
\end{align}
is a valid equation in any AC power flow model.  This property follows directly from \eqref{w_v_eq} as demonstrated by \eqref{w_soc_1}--\eqref{w_soc_3}.
The well-known second order cone constraint \eqref{w_soc} clearly provides a tight upper bound for \eqref{w_eq}.
The remaining question is how to develop a tight lower bound.  

We begin with Model \ref{model:svfs}, which includes a real number representation of \eqref{w_eq} and \eqref{w_9} plus the variable bounds.  Note that the bounds on $w^R_{ij}$ and $w^I_{ij}$ can be derived from Lemma \ref{lemma:w_egde_vars_bounds}.  

%
\begin{model}[t]
\caption{The Non-Convex Voltage Feasibility Set}
\label{model:svfs}
\begin{subequations}
\begin{align}
\mbox{\bf variables:} \nonumber \\
& w_{i}, w_{j} \mbox{ - voltage magnitude squared}  \nonumber \\
& w^R_{ij}, w^I_{ij}  \mbox{ - voltage product } \nonumber \\
\mbox{\bf subject to:} \nonumber \\
& (\bm {v^l}_i)^2 \leq w_i \leq (\bm {v^u}_i)^2  \label{eq:svfs_1} \\
& (\bm {v^l}_j)^2 \leq w_j \leq (\bm {v^u}_j)^2  \label{eq:svfs_2} \\
& \bm {w^{Rl}}_{ij} \leq w^R_{ij} \leq \bm {w^{Ru}}_{ij}  \label{eq:wr_b} \\
& \bm {w^{Il}}_{ij} \leq w^I_{ij} \leq \bm {w^{Iu}}_{ij}  \label{eq:wi_b} \\
& \tan(\bm {\theta^l}_{ij})w^R_{ij} \leq w^I_{ij} \leq \tan(\bm {\theta^u}_{ij})w^R_{ij} \label{eq:svfs_3}  \\
& (w^R_{ij})^2 + (w^I_{ij})^2 = w_i w_j \label{eq:svfs_4} 
\end{align}
\end{subequations}
\end{model}
The rest of this subsection is concerned with developing three valid inequalities for Model \ref{model:svfs}.
We first investigate the extreme points of the feasible region and then propose an {\em Extreme cut} based on the convex envelope of the quadratic function found in \eqref{w_eq}.
We then propose two valid convex {\em nonlinear cuts}, which are redundant in Model \ref{model:svfs}, but tighten its lifted convex relaxation. 

\subsubsection{An Illustrative Example}

Before developing analytical solutions, it is helpful to build intuition using an illustrative example. 
As presented, Model \ref{model:svfs} is defined over $(w^R_{ij}, w^I_{ij},w_{i},w_{j})  \in \rit^4$, which is not easy to visualize.
However, we observe that the nonlinear equation \eqref{eq:svfs_4} can be used to eliminate one of the variables, reducing the variable space to $\rit^3$.
We use $((w^R_{ij})^2 + (w^I_{ij})^2)/w_i = w_j$ to eliminate the $w_j$ variable and focus on the $(w^R_{ij}, w^I_{ij},w_{i})  \in \rit^3$ space.

\begin{figure}[t]
\center
     \begin{subfigure}{8.0cm}
        \includegraphics[width=8.0cm]{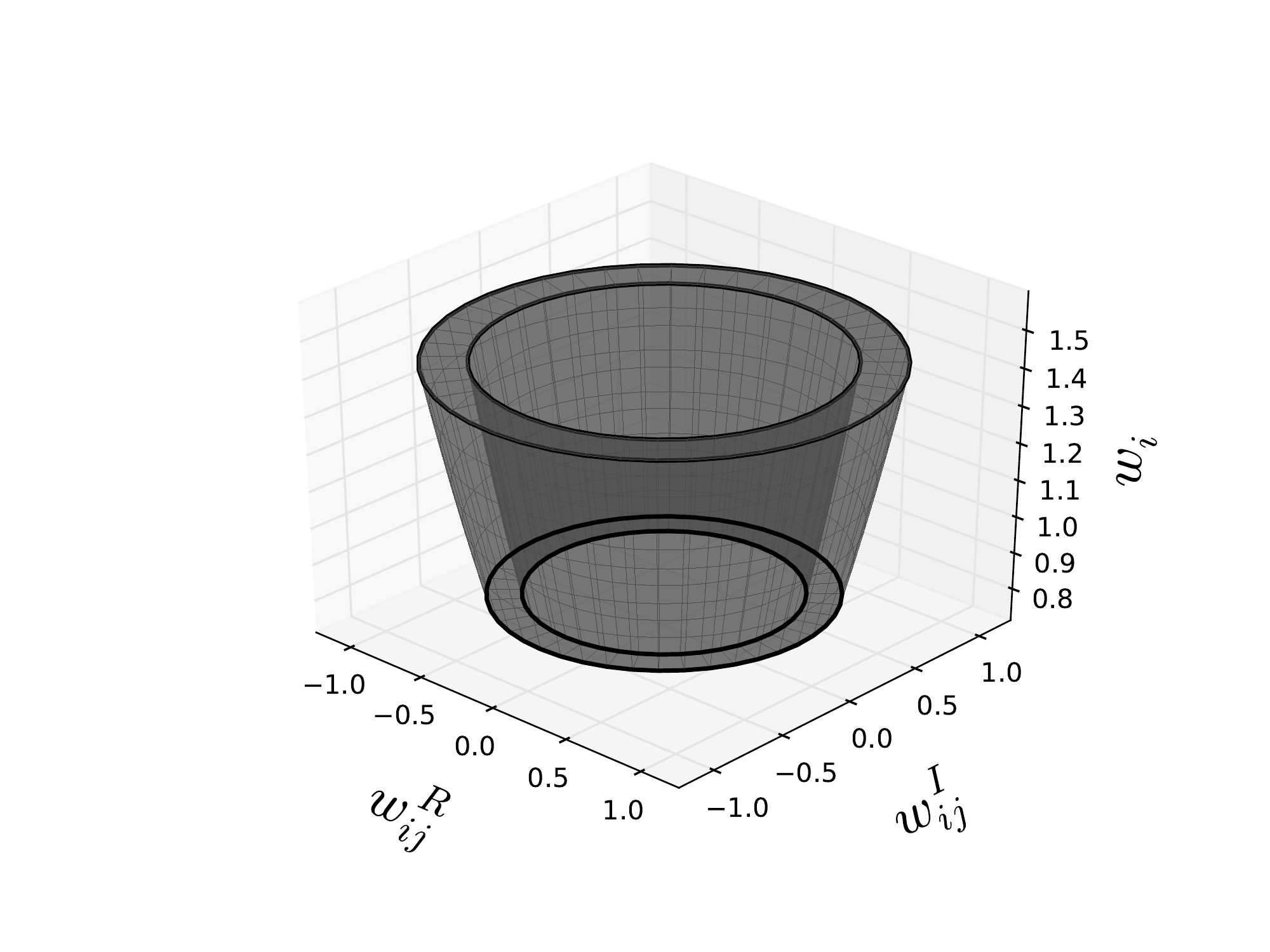} 
        \caption{Non-Convex Set without PAD Constraints.}
        \label{fig:w_set_std}
    \end{subfigure}
    \begin{subfigure}{8.0cm}
        \includegraphics[width=8.0cm]{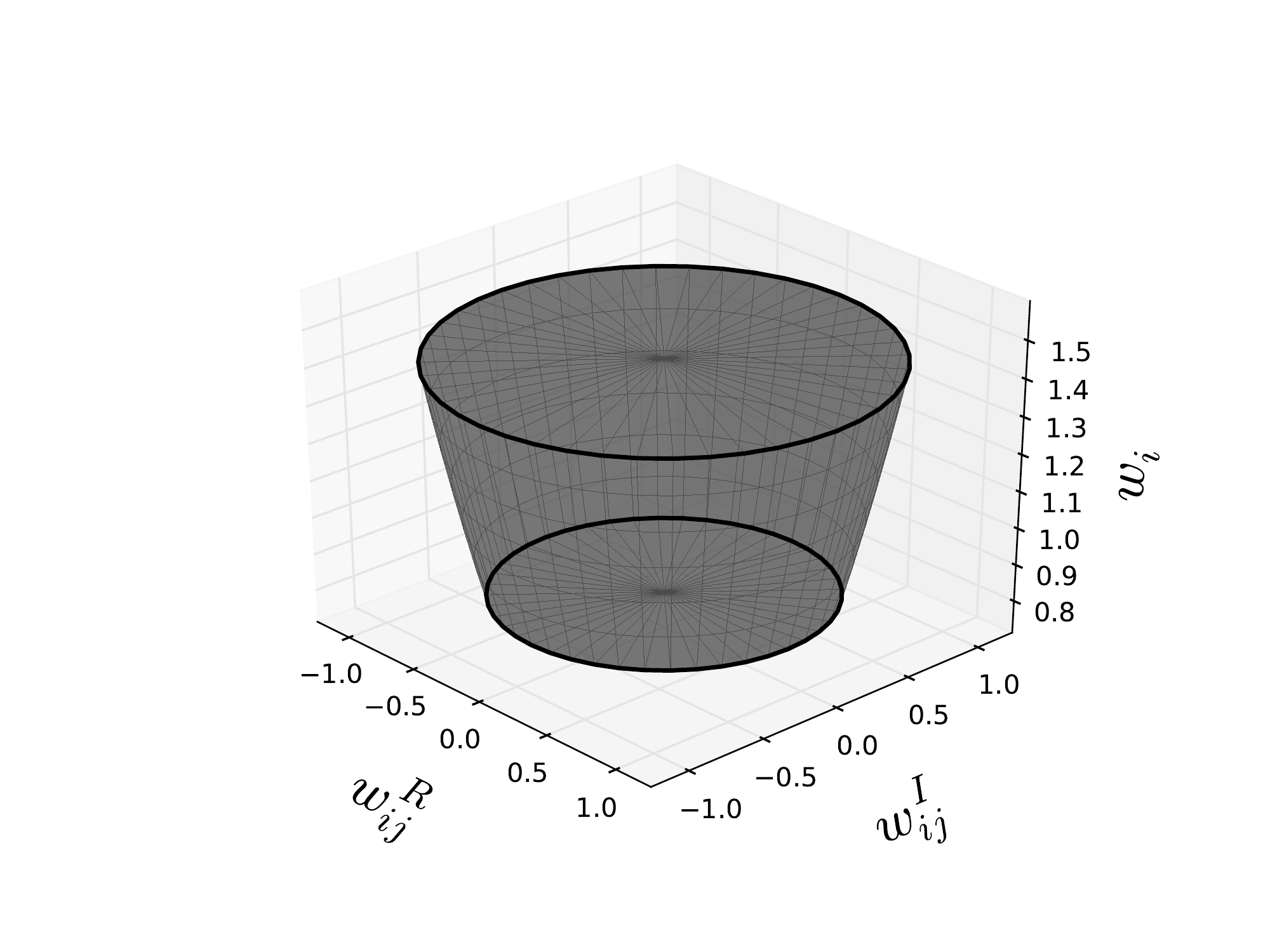} 
        \caption{Convex Hull without PAD Constraints.}
        \label{fig:w_set_std_cvx}
    \end{subfigure}  
    \begin{subfigure}{8.0cm}
        \includegraphics[width=8.0cm]{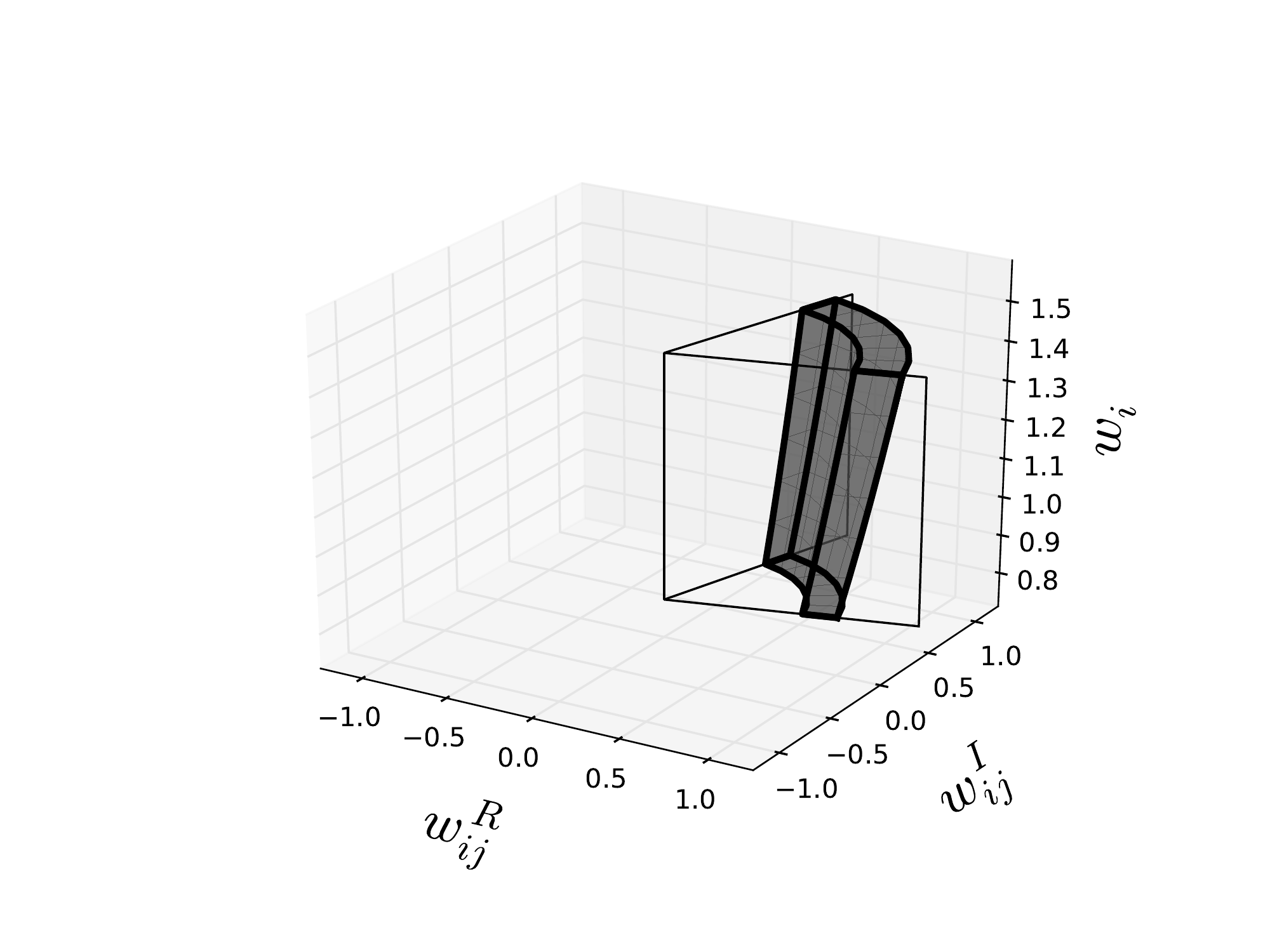} 
        \caption{Non-Convex Set with PAD Constraints}
        \label{fig:w_set_pad}
    \end{subfigure}  
    \begin{subfigure}{8.0cm}
        \includegraphics[width=8.0cm]{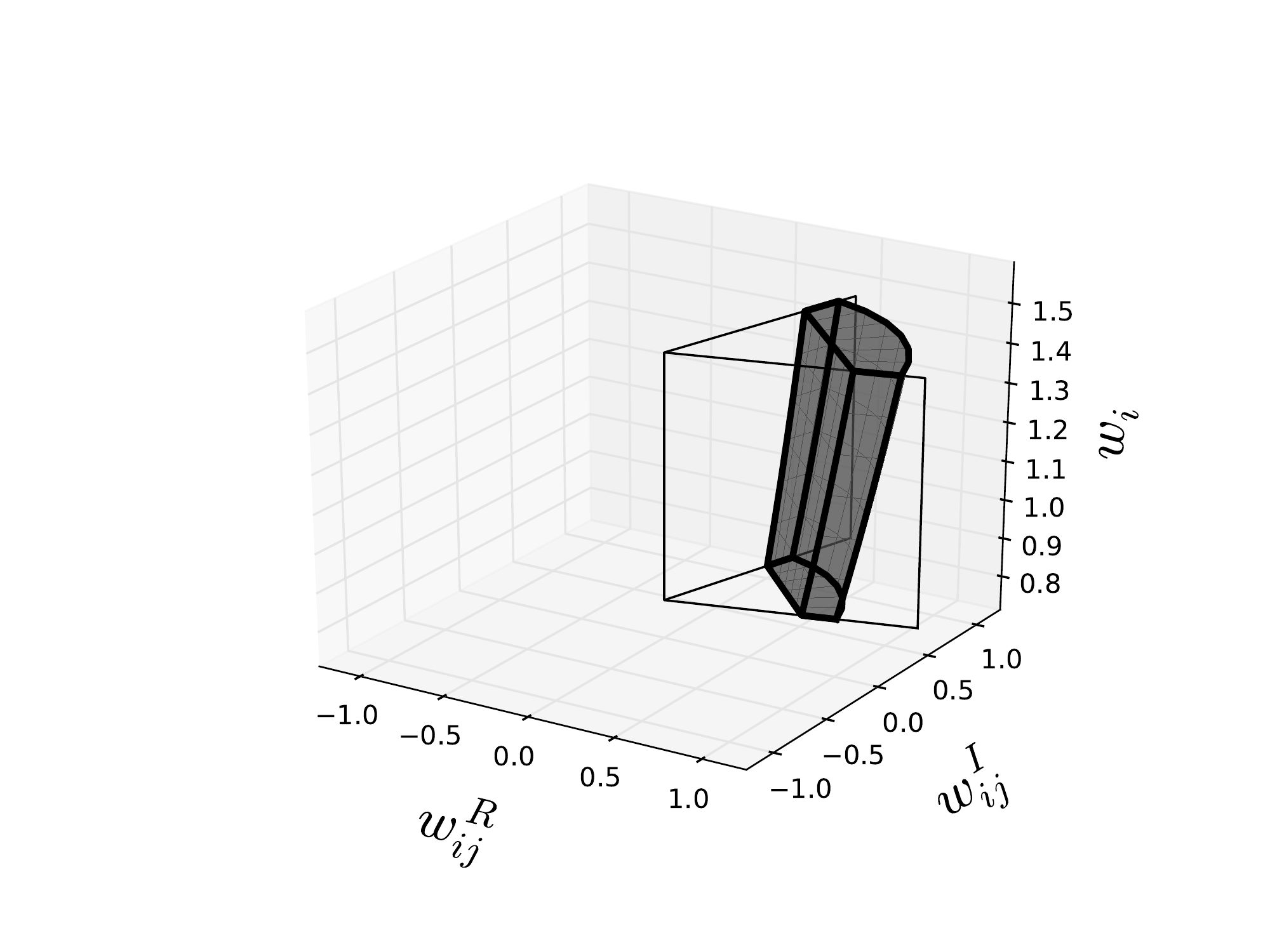} 
        \caption{Convex Hull with PAD Constraints.}
        \label{fig:w_set_pad_cvx}
    \end{subfigure}  
\caption{The Implications of PAD Constraints on the Convexification of \eqref{w_eq}.}
\label{fig:w_sets}
\end{figure}

let us consider Model \ref{model:svfs} with the parameters,
\begin{subequations}
\begin{align}
& \bm {v^l}_i = 0.9, \;\; \bm {v^u}_i = 1.2, \;\; \bm {v^l}_j = 0.8, \;\; \bm {v^u}_j = 1.0, \;\;  \bm {\theta^l}_{ij} = \bm \pi/12, \;\; \bm {\theta^u}_{ij} =  5 \bm \pi / 12 \nonumber 
\end{align}
\end{subequations}
Figure \ref{fig:w_sets} presents the solution set of Model \ref{model:svfs} with these parameters in the $(w^R_{ij}, w^I_{ij},w_{i})$ space.
This figure considers four cases, Model \ref{model:svfs} with and without the PAD constraint \eqref{eq:svfs_3} and the implications that this constraint has on the convexification of \eqref{w_eq}.
Figure \ref{fig:w_set_std} presents Model \ref{model:svfs} with only constraints on the voltage variables (i.e. \eqref{eq:svfs_1}--\eqref{eq:svfs_2},\eqref{eq:svfs_4}) and Figure \ref{fig:w_set_std_cvx}, illustrates the convex hull of that case.
Figure \ref{fig:w_set_pad} highlights the significant reduction in the feasible space when PAD constraints are considered (i.e.  \eqref{eq:svfs_1}--\eqref{eq:svfs_4}) and Figure \ref{fig:w_set_pad_cvx}, illustrates the much reduced convex hull.  
The next subsection develops an {\em Extreme cut} representing the analytical form of the convex hull illustrated in Figure \ref{fig:w_set_pad_cvx}.

\subsubsection{The Extreme Cut}


From this point forward, we use an alternate representation of the
voltage angle bounds.  Specifically, given $- \bm \pi/2 \leq \bm
{\theta^l}_{ij} < \bm {\theta^u}_{ij} \leq \bm \pi/2$, we define the following
constants:
\begin{subequations}
\begin{align}
& \bm \phi_{ij} = (\bm {\theta^u}_{ij} + \bm {\theta^l}_{ij})/2 \\
& \bm \delta_{ij} = (\bm {\theta^u}_{ij} - \bm {\theta^l}_{ij})/2
\end{align}
\end{subequations}
Observe that $\bm {\theta^l}_{ij} = \bm \phi_{ij}- \bm \delta_{ij}$
and $\bm {\theta^u}_{ij} = \bm \phi_{ij}+ \bm \delta_{ij}$.
Additionally, we define the following constants,
\begin{subequations}
\begin{align}
& \bm {v^\sigma}_i = \bm {v^l}_i + \bm {v^u}_i \label{eq:sec_1} \\
& \bm {v^\sigma}_j = \bm {v^l}_j + \bm {v^u}_j \label{eq:sec_2}
\end{align}
\end{subequations}
As this section demonstrates, the $\phi, \delta, v^\sigma$
representation is particularly advantageous for developing concise
valid inequalities for Model \ref{model:svfs}.

%

\begin{theorem}
The following Extreme cut is redundant in Model \ref{model:svfs},
\begin{equation}
\bm {v^l}_j \cos(\bm \delta_{ij}) w_i  - \bm {v^\sigma}_i \cos(\bm \phi_{ij})w^R_{ij} - \bm {v^\sigma}_i \sin(\bm \phi_{ij}) w^I_{ij} +  \bm {v^l}_i \bm {v^u}_i \bm {v^l}_j \cos(\bm \delta_{ij}) \leq 0. \label{eq:3d_cut}
\end{equation}
\end{theorem}
\begin{proof}
As mentioned previously, Model \ref{model:svfs} can be reformulated in three dimensions using equation \eqref{eq:svfs_4}, which leads to the set 
$$\mathcal S_p=  \left \{\left(w^R_{ij}, w^I_{ij},w_{i}\right)  \in \rit^3  ~\left |~ 
	\begin{aligned}
				&(\ref{eq:svfs_1}),(\ref{eq:wr_b})-(\ref{eq:svfs_3})\nonumber\\
				&w_i(\bm {v^l}_j)^2 \leq (w^R_{ij})^2 + (w^I_{ij})^2 \leq w_i(\bm {v^u}_j)^2  \label{eq:svfs_2_}
	\end{aligned} \right.
\right\}.$$
Let
\begin{align*}
&f(w^R_{ij}, w^I_{ij},w_{i}) = w_i(\bm {v^l}_j)^2 - (w^R_{ij})^2 - (w^I_{ij})^2,\\
&h(w^R_{ij}, w^I_{ij},w_{i})= \bm {v^l}_j \cos(\bm \delta_{ij}) w_i  - \bm {v^\sigma}_i \cos(\bm \phi_{ij})w^R_{ij} - \bm {v^\sigma}_i \sin(\bm \phi_{ij}) w^I_{ij} +  \bm {v^l}_i \bm {v^u}_i \bm {v^l}_j \cos(\bm \delta_{ij}),
\end{align*}
and define the set,
$$\mathcal S_r =  \left \{\left(w^R_{ij}, w^I_{ij},w_{i}\right)  \in \rit^3  ~\left |~ 
	\begin{aligned}
				&f(w^R_{ij}, w^I_{ij},w_{i}) \le 0,\nonumber\\
				&(\bm {v^l}_i)^2 \leq w_i \leq(\bm {v^u}_i)^2, w^R_{ij} \le \bm {v^u}_i\bm {v^u}_j\nonumber\\
				&\tan(\bm {\theta^l}_{ij})w^R_{ij} \leq w^I_{ij} \leq \tan(\bm {\theta^u}_{ij})w^R_{ij}
	\end{aligned} \right.
\right\},$$
observe that $\mathcal S_r$ is a relaxation of $\mathcal S_p$. We will first show that $h(w^R_{ij}, w^I_{ij},w_{i}) \le 0,~ \forall (w^R_{ij}, w^I_{ij},w_{i}) \in \mathcal S_r$, and consequently $\forall (w^R_{ij}, w^I_{ij},w_{i}) \in \mathcal S_p$, as $\mathcal S_p \subset \mathcal S_r$.
Consider the nonlinear program 
\begin{align}
&\max ~h(w^R_{ij}, w^I_{ij},w_{i}) \nonumber\\ 
&\text{ s.t. } (w^R_{ij}, w^I_{ij},w_{i}) \in \mathcal S_r. \label{LPRC} \tag{LPRC}
\end{align}

\eqref{LPRC} is a linear program with a reverse-convex constraint, or a concave budget constraint. Note that $\mathcal S_r$ is a bounded non-empty set and $f(w^R_{ij}, w^I_{ij},w_{i}) \le 0$ is a non-redundant constraint as it cuts the points satisfying $w^R_{ij} = w^I_{ij} = 0$. This type of problem is studied in \cite{Hillestad_74,Hillestad_80} where it is shown that all optimal solutions lie at the intersection of the concave constraint and the edges of the linear system (intersection of $n-1$ linear inequalities). There are only four such points in our case,
	\begin{align*}
& \mbox{point 1: }  w_i = (\bm {v^l}_i)^2, \;\; w^R_{ij} = \bm {v^l}_i \bm {v^l}_j \cos(\bm \phi_{ij} - \bm \delta_{ij}), \;\; w^I_{ij} = \bm {v^l}_i \bm {v^l}_j \sin(\bm \phi_{ij} - \bm \delta_{ij}) \nonumber \\
& \mbox{point 2: }  w_i = (\bm {v^l}_i)^2, \;\; w^R_{ij} = \bm {v^l}_i \bm {v^l}_j \cos(\bm \phi_{ij} + \bm \delta_{ij}), \;\; w^I_{ij} = \bm {v^l}_i \bm {v^l}_j \sin(\bm \phi_{ij} + \bm \delta_{ij}) \nonumber \\
& \mbox{point 3: }  w_i = (\bm {v^u}_i)^2, \;\; w^R_{ij} = \bm {v^u}_i \bm {v^l}_j \cos(\bm \phi_{ij} - \bm \delta_{ij}), \;\; w^I_{ij} = \bm {v^u}_i \bm {v^l}_j \sin(\bm \phi_{ij} - \bm \delta_{ij}) \nonumber \\
& \mbox{point 4: }  w_i = (\bm {v^u}_i)^2, \;\; w^R_{ij} = \bm {v^u}_i \bm {v^l}_j \cos(\bm \phi_{ij} + \bm \delta_{ij}), \;\; w^I_{ij} = \bm {v^u}_i \bm {v^l}_j \sin(\bm \phi_{ij} + \bm \delta_{ij}) \nonumber 
	\end{align*}
	all of which satisfy $f(w^R_{ij}, w^I_{ij},w_{i}) = h(w^R_{ij}, w^I_{ij},w_{i}) = 0$. Since zero is the maximizer of \eqref{LPRC}, it follows that $h(w^R_{ij}, w^I_{ij},w_{i}) \le 0,~ \forall (w^R_{ij}, w^I_{ij},w_{i}) \in \mathcal S_r$ and consequently $\forall (w^R_{ij}, w^I_{ij},w_{i}) \in \mathcal{S}_p$.

\end{proof}
\noindent
Given the valid linear cut \eqref{eq:3d_cut}, we can define a convex relaxation of $\mathcal S_p$,
$$\mathcal S_c=  \left \{(w^R_{ij}, w^I_{ij},w_{i})  \in \rit^3  ~\left |~ 
	\begin{aligned}
				& (\ref{eq:svfs_1}),(\ref{eq:wr_b})-(\ref{eq:svfs_3}),(\ref{eq:3d_cut}), \nonumber \\
				& (w^R_{ij})^2 + (w^I_{ij})^2 \leq w_i(\bm {v^u}_j)^2
	\end{aligned} \right.
\right \}.$$


\noindent
An example of $\mathcal S_p$ and $\mathcal S_c$ are presented in Figures \ref{fig:w_set_pad} and \ref{fig:w_set_pad_cvx} respectively.\\
Let us emphasize that projecting the feasible region of Model \ref{model:svfs} into the $(w^R_{ij}, w^I_{ij},w_{j})$ space can lead to a similar Extreme cut,
\begin{equation}
\bm {v^l}_i \cos(\bm \delta_{ij}) w_j  - \bm {v^\sigma}_j \cos(\bm \phi_{ij})w^R_{ij} - \bm {v^\sigma}_j \sin(\bm \phi_{ij}) w^I_{ij} +  \bm {v^l}_j \bm {v^u}_j \bm {v^l}_i \cos(\bm \delta_{ij}) \leq 0. \label{eq:3d_cut2}
\end{equation}

\subsubsection{The Convex Nonlinear Cuts}
Let us emphasize that the convex relaxation of Model \ref{model:svfs} lives in a four-dimensional space, while the Extreme cuts defined above are three-dimensional, excluding the variable $w_j$.
In this section, we utilize the convex set $\mathcal S_c$ to develop two valid four-dimensional cuts based on lifting redundant constraints in the $(w^R_{ij}, w^I_{ij},w_{i})  \in \rit^3$ space.
\paragraph{The VUB Nonlinear Cut}
For clarity we begin by defining the following constants,
\begin{subequations}
\begin{align}
& \bm c_{11} = \bm {v^\sigma}_i\bm {v^\sigma}_j \cos(\bm \phi_{ij})  \\
& \bm c_{12} = \bm {v^\sigma}_i\bm {v^\sigma}_j \sin(\bm \phi_{ij})  \\
& \bm c_{13} = - \bm {v^u}_j \cos(\bm \delta_{ij})\bm {v^\sigma}_j  \\
& \bm c_{14} = - \bm {v^u}_i \cos(\bm \delta_{ij})\bm {v^\sigma}_i   \\
& \bm c_{15} = - \bm {v^u}_i\bm {v^u}_j \cos(\bm \delta_{ij})(\bm {v^l}_i\bm {v^l}_j - \bm {v^u}_i\bm {v^u}_j)  
\end{align}
\end{subequations}

Consider the optimization problem,
\begin{align}
& \min ~g(w^R_{ij}, w^I_{ij},w_{i}) = \bm c_{11} w^R_{ij} + \bm c_{12} w^I_{ij} + \bm c_{13} w_i + \bm c_{14} \frac{(w^R_{ij})^2 + (w^I_{ij})^2}{w_i} + \bm c_{15} \nonumber \\
&\text{s.t. } \left \{\label{NLP}   \tag{NLP}   
\begin{aligned}           
	&(\bm {v^l}_i)^2 \leq w_i \leq(\bm {v^u}_i)^2,\\
	&\tan(\bm {\theta^l}_{ij})w^R_{ij} \leq w^I_{ij} \leq \tan(\bm {\theta^u}_{ij})w^R_{ij},\\
	&\bm {v^l}_j \cos(\bm \delta_{ij}) w_i  - \bm {v^\sigma}_i \cos(\bm \phi_{ij})w^R_{ij} - \bm {v^\sigma}_i \sin(\bm \phi_{ij}) w^I_{ij} +  \bm {v^l}_i \bm {v^u}_i \bm {v^l}_j \cos(\bm \delta_{ij}) \leq 0
\end{aligned}\right.
\end{align}
%
%
%

\begin{proposition} \label{th:zero}
The optimal objective for \eqref{NLP} is non-negative.
\end{proposition} 
\begin{proof}
In \cite{QCarchive}, Hijazi et al. prove that the function $f(x,y,z) = (x^2+y^2)/z, z >0$, is convex, thus \eqref{NLP} is a concave program as $\bm c_{14} < 0$.
Based on \cite{Benson95}, optimal solutions in \eqref{NLP} are extreme points of the feasibility region. There are four extreme points in \eqref{NLP},
\begin{align*}
& \mbox{point 1: }  w_i = (\bm {v^l}_i)^2, \;\; w^R_{ij} = \bm {v^l}_i \bm {v^l}_j \cos(\bm \phi_{ij} - \bm \delta_{ij}), \;\; w^I_{ij} = \bm {v^l}_i \bm {v^l}_j \sin(\bm \phi_{ij} - \bm \delta_{ij}) \nonumber \\
& \mbox{point 2: }  w_i = (\bm {v^l}_i)^2, \;\; w^R_{ij} = \bm {v^l}_i \bm {v^l}_j \cos(\bm \phi_{ij} + \bm \delta_{ij}), \;\; w^I_{ij} = \bm {v^l}_i \bm {v^l}_j \sin(\bm \phi_{ij} + \bm \delta_{ij}) \nonumber \\
& \mbox{point 3: }  w_i = (\bm {v^u}_i)^2, \;\; w^R_{ij} = \bm {v^u}_i \bm {v^l}_j \cos(\bm \phi_{ij} - \bm \delta_{ij}), \;\; w^I_{ij} = \bm {v^u}_i \bm {v^l}_j \sin(\bm \phi_{ij} - \bm \delta_{ij}) \nonumber \\
& \mbox{point 4: }  w_i = (\bm {v^u}_i)^2, \;\; w^R_{ij} = \bm {v^u}_i \bm {v^l}_j \cos(\bm \phi_{ij} + \bm \delta_{ij}), \;\; w^I_{ij} = \bm {v^u}_i \bm {v^l}_j \sin(\bm \phi_{ij} + \bm \delta_{ij}) \nonumber 
	\end{align*}
	all of which satisfy $g(w^R_{ij}, w^I_{ij},w_{i}) \ge 0$.
\end{proof} 

\begin{theorem} 
\label{thrm:3d_cut_1}
In the $(w^R_{ij}, w^I_{ij},w_{i})$ space, the following nonlinear cut is redundant with respect to $\mathcal S_p$.
\begin{subequations}
\begin{align}
& \bm c_{11} w^R_{ij} + \bm c_{12} w^I_{ij} + \bm c_{13} w_i + \bm c_{14} \frac{(w^R_{ij})^2 + (w^I_{ij})^2}{w_i} + \bm c_{15} \geq 0 \label{eq:3d_cut_1} 
\end{align}
\end{subequations}
\end{theorem}
\begin{proof}
\noindent
Since the feasibility space of \eqref{NLP} is a relaxation of $\mathcal S_c$, Proposition \ref{th:zero} implies that $$g(w^R_{ij}, w^I_{ij},w_{i}) \geq 0 ,\forall (w^R_{ij}, w^I_{ij},w_{i}) \in \mathcal S_c,$$ thus constraint  \eqref{eq:3d_cut_1} is redundant for $\mathcal S_c$ and consequently for the restricted set $\mathcal S_p$.
\end{proof}

\paragraph{The VLB Nonlinear Cut}
For clarity we begin by defining the following constants,
\begin{subequations}
\begin{align}
& \bm c_{21} = \bm {v^\sigma}_i\bm {v^\sigma}_j \cos(\bm \phi_{ij}) \nonumber \\
& \bm c_{22} = \bm {v^\sigma}_i\bm {v^\sigma}_j \sin(\bm \phi_{ij}) \nonumber \\
& \bm c_{23} = - \bm {v^l}_j \cos(\bm \delta_{ij})\bm {v^\sigma}_j \nonumber \\
& \bm c_{24} = - \bm {v^l}_i \cos(\bm \delta_{ij})\bm {v^\sigma}_i  \nonumber \\
& \bm c_{25} = \bm {v^l}_i\bm {v^l}_j \cos(\bm \delta_{ij})(\bm {v^l}_i\bm {v^l}_j - \bm {v^u}_i\bm {v^u}_j) \nonumber 
\end{align}
\end{subequations}

\begin{theorem}
In the $(w^R_{ij}, w^I_{ij},w_{i})$ space, the following nonlinear cut is redundant with respect to $\mathcal S_p$.
\begin{subequations}
\begin{align}
& \bm c_{21} w^R_{ij} + \bm c_{22} w^I_{ij} + \bm c_{23} w_i + \bm c_{24} \frac{(w^R_{ij})^2 + (w^I_{ij})^2}{w_i} + \bm c_{25} \geq 0 \label{eq:3d_cut_2}
\end{align}
\end{subequations}
\end{theorem}
\begin{proof}
The proof of Theorem \ref{thrm:3d_cut_1} can be adapted to fit the new parameters introduced here.
\end{proof}

\begin{corollary}
Constraints \eqref{eq:3d_cut_1} and \eqref{eq:3d_cut_2} are valid nonlinear inequalities in any power flow model or power flow relaxation.  
\end{corollary}

\subsubsection{Application of the Valid Inequalities}

The usefulness of the nonlinear cuts \eqref{eq:3d_cut_1} and
\eqref{eq:3d_cut_2} is not immediately clear.  Indeed, the Extreme cut
\eqref{eq:3d_cut} appears to provide the tightest convex relaxation of
the three-dimensional non-convex set defined in Model
\ref{model:svfs}.  However, it is important to point out that as soon
as we relax the quadratic equation \eqref{eq:svfs_4}, we lift the
feasible region into four dimensions, that is $(w^R_{ij},
w^I_{ij},w_{i},w_{j}) \in \rit^4$.  The key insight is that although
\eqref{eq:3d_cut_1} and \eqref{eq:3d_cut_2} are redundant in the
three-dimensional space, they are not redundant in the lifted $\rit^4$
space.  This property was observed in \cite{6822653}, where a
collection of line flow constraints, which are equivalent in the
non-convex space, were shown to have different strengths in the lifted
convex relaxation space.  Utilizing the equivalence $((w^R_{ij})^2 +
(w^I_{ij})^2) / w_i = w_j$, we can lift \eqref{eq:3d_cut_1} and
\eqref{eq:3d_cut_2} into the standard $\rit^4$ power flow relaxation
space as follows,
\begin{subequations}
\begin{align}
& \bm {v^\sigma}_i\bm {v^\sigma}_j(w^R_{ij}\cos(\bm \phi_{ij}) + w^I_{ij}\sin(\bm \phi_{ij})) - \bm {v^u}_j \cos(\bm \delta_{ij})\bm {v^\sigma}_jw_i - \bm {v^u}_i \cos(\bm \delta_{ij})\bm {v^\sigma}_iw_j \geq  \bm {v^u}_i\bm {v^u}_j \cos(\bm \delta_{ij})(\bm {v^l}_i\bm {v^l}_j - \bm {v^u}_i\bm {v^u}_j) \label{eq:4d_cut_1} \\
& \bm {v^\sigma}_i\bm {v^\sigma}_j(w^R_{ij}\cos(\bm \phi_{ij}) + w^I_{ij}\sin(\bm \phi_{ij})) - \bm {v^l}_j \cos(\bm \delta_{ij})\bm {v^\sigma}_jw_i - \bm {v^l}_i \cos(\bm \delta_{ij})\bm {v^\sigma}_iw_j \geq -\bm {v^l}_i\bm {v^l}_j \cos(\bm \delta_{ij})(\bm {v^l}_i\bm {v^l}_j - \bm {v^u}_i\bm {v^u}_j) \label{eq:4d_cut_2} 
\end{align}
\end{subequations}
We refer to these constraints as {\em lifted nonlinear cuts} (LNC).  
Noting that these constraints are linear in the $\rit^4$ space, they can be easily integrated into any of the models discussed in Section \ref{sec:relaxations}.
\begin{proposition}
The  LNC cuts dominate the Extreme cuts in the $(w^R_{ij}, w^I_{ij},w_{i},w_{j})$ space.
\end{proposition}
\begin{proof}
Observe that replacing $w_j$ (resp. $w_i$) by its lower bound in \eqref{eq:4d_cut_2} (resp. \eqref{eq:4d_cut_1}) leads to \eqref{eq:3d_cut} (resp. \eqref{eq:3d_cut2}). Given that the coefficients corresponding to $w_i$ and $w_j$ are both negative in \eqref{eq:4d_cut_1} and \eqref{eq:4d_cut_2}, dominance is guaranteed. 
\end{proof}

\subsubsection{Connections to Previous Work}

To the best of our knowledge, two previous work in the power systems community \cite{6822653,7056568}
have explored similar ideas for strengthening the SDP relaxation.  Two
interesting observations were made in \cite{6822653}: (1) when the
voltage magnitudes at both sides of the line are fixed, the maximum
phase difference $\bm {\theta^m}_{ij}$ can be used to encode a variety
of equivalent line capacity constraints; (2) from these equivalent
flow limit constraints, the current limit constraint was observed to
be the most advantageous for the SDP relaxation.  Specifically, in the
notation of this paper, \cite{6822653} concludes that for the
intervals $w_i = 1, w_j = 1, \bm {\theta^l}_{ij} = -\bm
{\theta^m}_{ij}, \bm {\theta^u}_{ij} = \bm {\theta^m}_{ij}$, the
strongest line flow constraint in the SDP relaxation is $w_i + w_j -
2w^R_{ij} \leq 2(1 - \cos(\bm {\theta^m}_{ij}))$.  Knowing that the
values of $w_i,w_j$ are fixed, this constraint reduces to:
\begin{align}
w^R_{ij} \geq \cos(\bm {\theta^m}_{ij})
\end{align}
Now let us apply the same special case to the lifted nonlinear cuts developed here.  The constants for this special case are $\bm \phi_{ij} = 0;\; \bm \delta_{ij} = \bm {\theta^m}_{ij};\; \bm {v^\sigma}_i, \bm {v^\sigma}_j = 2;\; \bm {v^l}_i,\bm {v^u}_i,\bm {v^l}_j,\bm {v^u}_i = 1$ and the application to  \eqref{eq:4d_cut_1} is as follows:\footnote{In this particular case, \eqref{eq:4d_cut_2} yields an identical result.}
\begin{subequations}
\begin{align}
\bm {v^\sigma}_i\bm {v^\sigma}_j(w^R_{ij}\cos(\bm \phi_{ij}) + w^I_{ij}\sin(\bm \phi_{ij})) - \bm {v^u}_j \cos(\bm \delta_{ij})\bm {v^\sigma}_jw_i - \bm {v^u}_i \cos(\bm \delta_{ij})\bm {v^\sigma}_iw_j &\geq  \bm {v^u}_i\bm {v^u}_j \cos(\bm \delta_{ij})(\bm {v^l}_i\bm {v^l}_j - \bm {v^u}_i\bm {v^u}_j) \\
4w^R_{ij} - \cos(\bm {\theta^m}_{ij})2w_i - \cos(\bm {\theta^m}_{ij})2w_j &\geq 0 \\
2w^R_{ij} - \cos(\bm {\theta^m}_{ij}) (w_i + w_j) &\geq 0 \\
w^R_{ij} &\geq \cos(\bm {\theta^m}_{ij})
\end{align}
\end{subequations}
This reduction shows that the lifted nonlinear cuts proposed here are a generalization dominating the current limit constraint proposed in \cite{6822653}.

In an entirely different approach, valid cuts based on the bounds of
$w^R$ and $w^I$ were proposed in \cite{7056568}.  These cuts have a
key advantage over the line limit constraints considered in
\cite{6822653} in that they can capture the structure of asymmetrical
bounds on $\bm {\theta^l}, \bm {\theta^u}$.  For example, 
consider the case where $0 \leq \bm {\theta^l}_{ij} < \bm
{\theta^u}_{ij} \leq \bm \pi / 2$.  In this case, \cite{7056568} proposes the following cut,
\begin{subequations}
\begin{align}
w^R_{ij}\cos(\bm \phi_{ij}) + w^I_{ij}\sin(\bm \phi_{ij}) \geq \bm {v^l}_i\bm {v^l}_j \cos(\bm \delta_{ij}) \label{eq:w_bound_cut}
\end{align}
\end{subequations}
A derivation of this cut from the algorithm provided in \cite{7056568} can be found in Appendix \ref{sec:w_bound_cut}.  
\begin{proposition}
The new nonlinear lifted cuts \eqref{eq:4d_cut_2} dominate constraints \eqref{eq:w_bound_cut}.
\end{proposition}
\begin{proof}
To support the proof, we first observe the following property,
\begin{subequations}
\begin{align}
& \bm {v^l}_j \bm {v^\sigma}_j (\bm {v^l}_i)^2 +  \bm {v^l}_i \bm {v^\sigma}_i (\bm {v^l}_j)^2 - \bm {v^l}_i\bm {v^l}_j (\bm {v^l}_i\bm {v^l}_j - \bm {v^u}_i\bm {v^u}_j) = \\
& (\bm {v^l}_j)^2 (\bm {v^l}_i)^2 + \bm {v^l}_j \bm {v^u}_j (\bm {v^l}_i)^2 +  (\bm {v^l}_i)^2 (\bm {v^l}_j)^2 +  \bm {v^l}_i \bm {v^u}_i (\bm {v^l}_j)^2  - (\bm {v^l}_i\bm {v^l}_j)^2  + \bm {v^l}_i\bm {v^l}_j \bm {v^u}_i\bm {v^u}_j = \\
& (\bm {v^l}_i \bm {v^l}_j + \bm {v^l}_i \bm {v^u}_j + \bm {v^u}_i \bm {v^l}_j + \bm {v^u}_i\bm {v^u}_j) \bm {v^l}_i \bm {v^l}_j = \\
& (\bm {v^l}_i + \bm {v^u}_i)(\bm {v^l}_j + \bm {v^u}_j) \bm {v^l}_i \bm {v^l}_j  =  \\
& \bm {v^\sigma}_i \bm {v^\sigma}_j \bm {v^l}_i \bm {v^l}_j 
\end{align}
\end{subequations}
Now assume $w_i = (\bm {v^l}_i)^2, w_j = (\bm {v^l}_j)^2$ and apply
\eqref{eq:4d_cut_2} as follows,
\begin{subequations}
\begin{align}
\bm {v^\sigma}_i\bm {v^\sigma}_j(w^R_{ij}\cos(\bm \phi_{ij}) + w^I_{ij}\sin(\bm \phi_{ij})) - \bm {v^l}_j \cos(\bm \delta_{ij})\bm {v^\sigma}_jw_i - \bm {v^l}_i \cos(\bm \delta_{ij})\bm {v^\sigma}_iw_j & \geq -\bm {v^l}_i\bm {v^l}_j \cos(\bm \delta_{ij})(\bm {v^l}_i\bm {v^l}_j - \bm {v^u}_i\bm {v^u}_j) \\
\bm {v^\sigma}_i\bm {v^\sigma}_j(w^R_{ij}\cos(\bm \phi_{ij}) + w^I_{ij}\sin(\bm \phi_{ij})) - \bm {v^l}_j \cos(\bm \delta_{ij})\bm {v^\sigma}_j (\bm {v^l}_i)^2 - \bm {v^l}_i \cos(\bm \delta_{ij})\bm {v^\sigma}_i (\bm {v^l}_j)^2  & \geq -\bm {v^l}_i\bm {v^l}_j \cos(\bm \delta_{ij})(\bm {v^l}_i\bm {v^l}_j - \bm {v^u}_i\bm {v^u}_j) \\
\bm {v^\sigma}_i\bm {v^\sigma}_j(w^R_{ij}\cos(\bm \phi_{ij}) + w^I_{ij}\sin(\bm \phi_{ij})) & \geq  \bm {v^\sigma}_i \bm {v^\sigma}_j \bm {v^l}_i \bm {v^l}_j \cos(\bm \delta_{ij}) \\
w^R_{ij}\cos(\bm \phi_{ij}) + w^I_{ij}\sin(\bm \phi_{ij}) & \geq \bm {v^l}_i \bm {v^l}_j \cos(\bm \delta_{ij})
\end{align}
\end{subequations}
A similar analysis can be done to confirm that \eqref{eq:w_bound_cut} is a weaker version of the extreme cut \eqref{eq:3d_cut}.
It is now clear that the cut proposed in \cite{7056568} is a special case of the cuts proposed here, where the voltage variables are assigned to their lower bounds.
\end{proof}

In a very recent and independent line of work, coming out of the mathematical programming community, \cite{chen_qcqp} considers a model similar to Model \ref{model:svfs}.
The key difference being in the parameterization the variable bounds and the coefficients of \eqref{eq:svfs_3}.
Using a representation where $\tan(\bm {\theta^l}_{ij}) = \bm {t^l}_{ij}, \tan(\bm {\theta^u}_{ij}) = \bm {t^u}_{ij}, (\bm {v^l}_{i})^2 = \bm {w^l}_{i}$, and so on, \cite{chen_qcqp} proposes the following constants,\footnote{This presentation ignores the special cases where $\bm {t^l}_{ij} = 0$ or $\bm {t^u}_{ij} = 0$.}
\begin{subequations}
\begin{align}
& \bm \pi_{0} = -\sqrt{\bm {w^l}_{i} \bm {w^l}_{j} \bm {w^u}_{i} \bm {w^u}_{j}} \label{eq:cc_const_1} \\
& \bm \pi_{1} = -\sqrt{\bm {w^l}_{j} \bm {w^u}_{j}} \\
& \bm \pi_{2} = -\sqrt{\bm {w^l}_{i} \bm {w^u}_{i}} \\
& \bm \pi_{3} = \left( \sqrt{\bm {w^l}_{i}} + \sqrt{\bm {w^u}_{i}} \right) \left( \sqrt{\bm {w^l}_{j}} + \sqrt{\bm {w^u}_{j}} \right) \frac{1 - \left( \frac{\sqrt{1+(\bm {t^l}_{ij})^2} -1}{\bm {t^l}_{ij}} \right) \left( \frac{\sqrt{1+(\bm {t^u}_{ij})^2} -1}{\bm {t^u}_{ij}} \right)}{1 + \left( \frac{\sqrt{1+(\bm {t^l}_{ij})^2} -1}{\bm {t^l}_{ij}} \right) \left( \frac{\sqrt{1+(\bm {t^u}_{ij})^2} -1}{\bm {t^u}_{ij}} \right)} \\
& \bm \pi_{4} = \left( \sqrt{\bm {w^l}_{i}} + \sqrt{\bm {w^u}_{i}} \right) \left( \sqrt{\bm {w^l}_{j}} + \sqrt{\bm {w^u}_{j}} \right) \frac{\left( \frac{\sqrt{1+(\bm {t^l}_{ij})^2} -1}{\bm {t^l}_{ij}} \right) + \left( \frac{\sqrt{1+(\bm {t^u}_{ij})^2} -1}{\bm {t^u}_{ij}} \right)}{1 + \left( \frac{\sqrt{1+(\bm {t^l}_{ij})^2} -1}{\bm {t^l}_{ij}} \right) \left( \frac{\sqrt{1+(\bm {t^u}_{ij})^2} -1}{\bm {t^u}_{ij}} \right)} \label{eq:cc_const_5}
\end{align}
\end{subequations}
%
and then develops the following valid inequalities,
\begin{subequations}
\begin{align}
& \bm \pi_{0} + \bm \pi_{1} w_i + \bm \pi_{2} w_j + \bm \pi_{3} w^R_{ij} + \bm \pi_{4} w^I_{ij} \geq \bm {w^u}_{j} w_i + \bm {w^u}_{i} w_j - \bm {w^u}_{i} \bm {w^u}_{j} \label{eq:cc_cut_1} \\
& \bm \pi_{0} + \bm \pi_{1} w_i + \bm \pi_{2} w_j + \bm \pi_{3} w^R_{ij} + \bm \pi_{4} w^I_{ij} \geq \bm {w^l}_{j} w_i + \bm {w^l}_{i} w_j - \bm {w^l}_{i} \bm {w^l}_{j} \label{eq:cc_cut_2} 
\end{align}
\end{subequations}
\begin{proposition}
Using the parameterization of Model \ref{model:svfs}, the valid inequalities \eqref{eq:cc_cut_1},\eqref{eq:cc_cut_2} are equivalent to \eqref{eq:4d_cut_1},\eqref{eq:4d_cut_2}, respectively.
\end{proposition}
\noindent
A proof can be found in Appendix \ref{sec:qcqp_cut}.
%

This result highlights how the transcendental characterization of the constant values (e.g. $\cos(\bm \phi_{ij})$, $\cos(\bm \delta_{ij})$, $\tan(\bm {\theta^l}_{ij})$, ...) used in Model \ref{model:svfs} simplifies the presentation of these valid inequalities as well as the proofs of their validity.

Together, all of these connections illustrate that the lifted nonlinear cuts proposed here and the valid inequalities from \cite{chen_qcqp} are a generalization of the cuts proposed in \cite{6822653} and \cite{7056568} that combines the strengths of both previous works.

\subsection{Bound Tightening}

It was observed in \cite{cp_qc_fp} that both the SDP and QC models benefit significantly from tightening the bounds on $v_i$ and $\theta_{ij}$.
Additionally, the convex envelopes of the QC model and all of the cuts proposed here also benefit form tight bounds.
Hence, we utilize the minimal network consistency algorithm proposed in \cite{cp_qc_fp} to strengthen all of the relaxations considered here.

\subsection{Impact on Model Size}
\label{sec:model_size}

This section has introduced a variety of methods for strengthening the SDP relaxation (i.e. Model \ref{model:ac_opf_w_sdp}), including adding the QC model constraints and/or lifted nonlinear cuts.
It is important to take note of the model size implications of each of these approaches.  
The  lifted nonlinear cuts are a notably light-weight improvement to the SDP relaxation and only require adding $2|E|$ linear constraints, and no additional variables.  
The QC constraints increase the model's size significantly and require adding $2|V| + 5|E|$ variables, $1 + |V| + 15|E|$ linear constraints, and $|V| + |E|$ quadratic constraints.
Consequently, one would expect the QC model to be stronger than the lifted nonlinear cuts but at the cost of a significant computation burden.

\section{Experimental Evaluation}
\label{sec:experiments}

This section assesses the benefits of all three SDP strengthening approaches in a step-wise fashion.
The assessment is done by comparing four variants of the SDP relaxation for bounding primal
AC-OPF solutions produced by IPOPT, which only guarantees local optimality. 
The four relaxations under consideration are as follows:
\begin{enumerate}
\item SDP-N : the SDP relaxation strengthened with the bound tightening proposed in \cite{cp_qc_fp}.
\item SDP-N+LNC : SDP-N with the addition of lifted nonlinear cuts.
\item SDP-N+QC : SDP-N with the conjunction of the QC model.
\item SDP-N+QC+LNC : SDP-N with the QC model and lifted nonlinear cuts.
\end{enumerate}
%

\paragraph*{Experimental Setting}

All of the computations are conducted on \amdqtwo. IPOPT 3.12
\cite{Ipopt} with linear solver ma27 \cite{hsl_lib}, as suggested by \cite{acopf_solvers}, was used as a
heuristic for finding locally optimal feasible solutions to the non-convex AC-OPF formulated in AMPL \cite{ampl}. 
The SDP relaxations were based on
the state-of-the-art implementation \cite{opfBranchDecompImpl} which
uses a branch decomposition \cite{opfBranchDecomp} for performance and scalability gains.  
The SDP solver SDPT3 4.0 \cite{Toh99sdpt3} was used with the modifications suggested in \cite{opfBranchDecompImpl}.
The tight variable bounds for SDP-N are pre-computed using the algorithm in \cite{cp_qc_fp}.  
If all of the subproblems are computed in parallel, the bound tightening computation adds an overhead of less than 1 minute, which is not reflected in the runtime results presented here.

\paragraph*{Open Test Cases}

Due to the computational burden of using modern SDP solvers on cases with more than 1000-buses \cite{qc_opf_tps}, the evaluation was conducted on 71 test cases from NESTA v0.6.0 \cite{nesta} that have less than 1000-buses.
Among these 71 test cases it was observed that the base case, SDP-N, was able to close the optimality gap to less than 1.0\% in 55 cases, leaving 16 open test cases.
Hence, we focus our attention on those test cases where the SDP-N optimality gap is greater than 1.0\%. Detailed performance and runtime results are present in Table \ref{tbl:gaps_time} and can be summarized as follows: 
\begin{enumerate}
\item SDP-N+LNC brings significant improvements to the SDP-N relaxation, often reducing the optimality gap by several percentage points.
\item SDP-N+QC is generally stronger than SDP-N+LNC, however nesta\_case162\_ieee\_dtc\_\_sad, \\{nesta\_case9\_na\_cao\_\_nco}, nesta\_case9\_nb\_cao\_\_nco are notable exceptions, illustrating that there is value in adding both the QC model and the lifted nonlinear cuts to the SDP relaxation.
\item The strongest model, SDP-N+QC+LNC, has reduced to optimality gap of 8 of the 16 of the open cases to less than 1\% (i.e. closing 50\% of the open cases), leaving only 8 for further investigation.  Furthermore, on 3 of the 8 open cases, the AC solution is known to be globally optimal, indicating that the only source of the optimality gap comes from convexificaiton.  These cases are ideal candidates for evaluation of nonconvex optimization algorithms.
\item Although the size of the SDP-N+QC model is significantly larger than SDP-N+LNC (as discussed in Section \ref{sec:model_size}), we observe that the runtimes do not vary significantly.  We suspect that the SDP iteration computation dominates the runtime on the test cases considered here.
\end{enumerate}

\begin{table*}[t]
\scriptsize
\center
\caption{Quality and Runtime Results of AC Power Flow Relaxations (open cases)}
\begin{tabular}{|r||r||r|r|r|r||r|r|r|r|r|r|r|r|r|c|c|}
\hline
                 & \$/h & \multicolumn{4}{c||}{Optimality Gap (\%)} & \multicolumn{5}{c|}{Runtime (seconds)} \\
                 &       &             &           &          & +QC   &      &              &           &         & +QC \\
Test Case & AC & SDP-N & +LNC & +QC & +LNC & AC & SDP-N & +LNC & +QC & +LNC \\
\hline
\hline
\multicolumn{11}{|c|}{Typical Operating Conditions (TYP)} \\
\hline
 nesta\_case5\_pjm &  {\bf 17551.89} & 5.22 & 5.06 & 3.96 & 3.96 & 0.16 & 3.18 & 2.92 & 3.36 & 3.04 \\
\hline
\hline
\multicolumn{11}{|c|}{Congested Operating Conditions (API)} \\
\hline
 nesta\_case30\_fsr\_\_api & 372.14 & 3.58 & 1.03 & 0.89 & 0.61 & 0.09 & 3.63 & 5.38 & 5.46 & 5.56 \\
\hline
 nesta\_case89\_pegase\_\_api & 4288.02 & 18.11 & 18.08$^\star$ & 17.09$^\star$ & 16.60$^\star$ & 0.50 & 12.44 & 13.17 & 47.50 & 28.27 \\
\hline
 nesta\_case118\_ieee\_\_api & 10325.27 & 16.72 & 8.70 & 3.40 & 3.32 & 0.40 & 8.49 & 9.61 & 10.73 & 13.62 \\
\hline
\hline
\multicolumn{11}{|c|}{Small Angle Difference Conditions (SAD)} \\
\hline
 nesta\_case24\_ieee\_rts\_\_sad & 79804.30 & 1.38 & 0.05 & 0.07 & 0.02 & 0.20 & 3.80 & 4.13 & 3.27 & 3.77 \\
\hline
 nesta\_case29\_edin\_\_sad & 46931.74 & 5.79 & 1.90 & 0.53 & 0.50 & 0.35 & 4.70 & 5.34 & 6.23 & 6.11 \\
\hline
 nesta\_case73\_ieee\_rts\_\_sad & 235241.58 & 2.41 & 0.18 & 0.05 & 0.03 & 0.26 & 6.44 & 6.80 & 8.01 & 8.51 \\
\hline
 nesta\_case118\_ieee\_\_sad & 4324.17 & 4.04 & 1.16 & 0.83 & 0.74 & 0.32 & 11.21 & 10.44 & 11.65 & 14.31 \\
\hline
 nesta\_case162\_ieee\_dtc\_\_sad & 4369.19 & 1.73 & 0.37 & 1.49 & 0.35 & 0.68 & 20.16 & 20.18 & 53.54 & 40.58 \\
\hline
 nesta\_case189\_edin\_\_sad & 914.64 & 1.20$^\star$ & 0.89$^\star$ & err. & 0.86$^\star$ & 0.29 & 7.51 & 10.91 & 36.24$^\star$ & 54.44 \\
\hline
\hline
\multicolumn{11}{|c|}{Nonconvex Optimization Cases (NCO)} \\
\hline
 nesta\_case9\_na\_cao\_\_nco & {\bf -212.43} & 18.00 & 11.66 & 15.91 & 11.62 & 0.05 & 2.42 & 2.66 & 3.26 & 2.30 \\
\hline
 nesta\_case9\_nb\_cao\_\_nco & {\bf -247.42} & 19.23 & 11.77 & 16.46 & 11.76 & 0.18 & 2.44 & 2.55 & 2.06 & 2.31 \\
\hline
 nesta\_case14\_s\_cao\_\_nco & 9670.44 & 2.96 & 2.92 & 2.06 & 2.03 & 0.07 & 3.21 & 2.86 & 2.91 & 3.10 \\
\hline
\hline
\multicolumn{11}{|c|}{Radial Toplogies (RAD)} \\
\hline
 nesta\_case9\_kds\_\_rad & 11279.48 & 1.09 & 0.13 & 1.04$^\star$ & 0.13 & 0.29 & 2.47 & 2.34 & 2.08 & 2.54 \\
\hline
 nesta\_case30\_kds\_\_rad & {\bf 4336.18}$^\dagger$ & 2.11 & 1.88 & 1.97 & 1.88 & n.a. & 4.02 & 3.25 & 6.51 & 3.84 \\
\hline
 nesta\_case30\_l\_kds\_\_rad & {\bf 3607.73}$^\dagger$ & 15.86 & 15.56 & 15.76 & 15.56 & n.a. & 3.53 & 3.28 & 4.69 & 4.26 \\
\hline
%
\end{tabular}\\
{\bf bold} - known global optimum, $\dagger$ - best known solution (not initial ipopt solution), $\star$ - solver reported numerical accuracy warnings.
\label{tbl:gaps_time}
\end{table*}

\begin{figure}[t!]
\center
    \includegraphics[width=6.0cm]{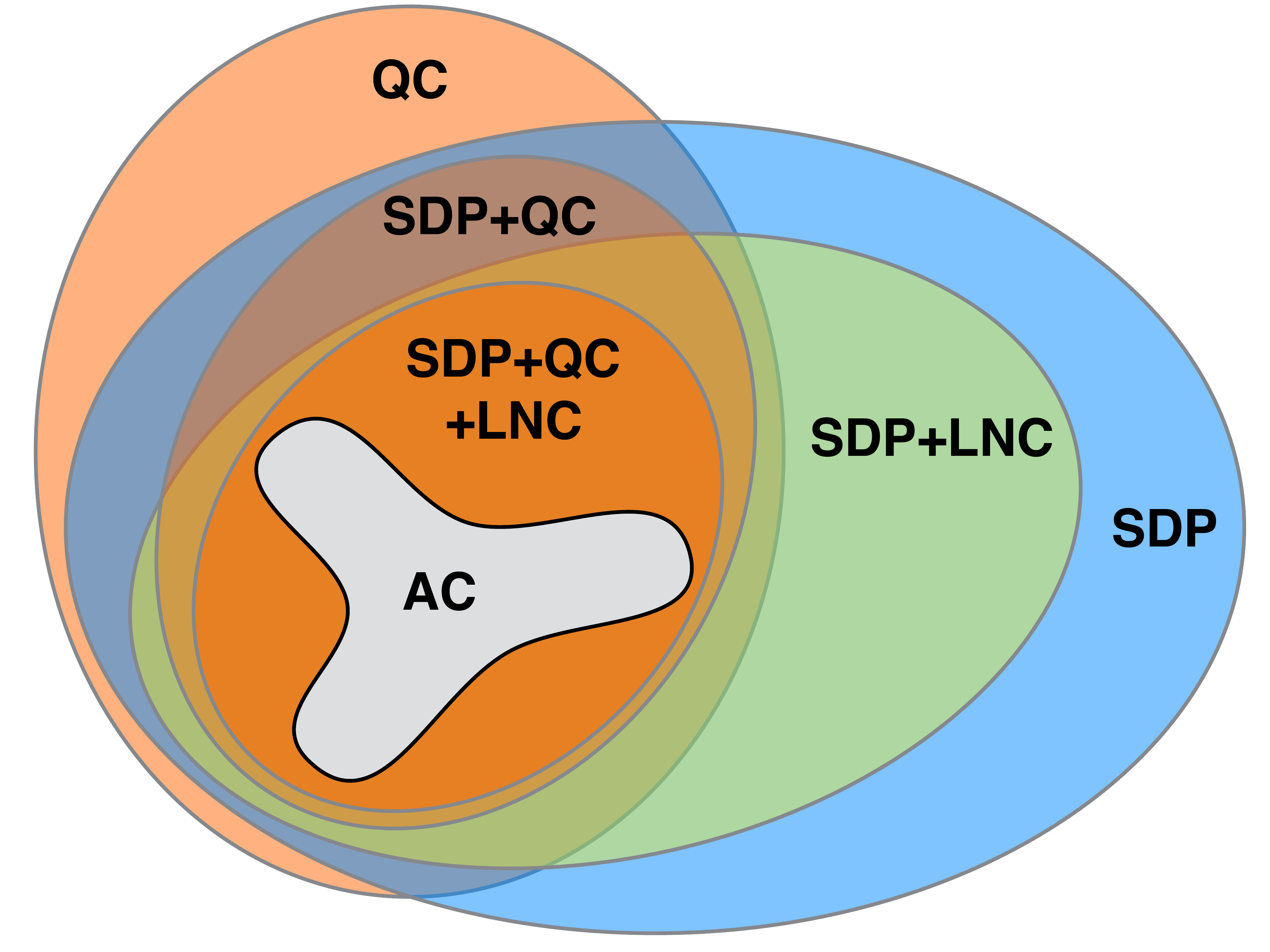} 
\caption{A Venn Diagram of the Solutions Sets for Various SDP Relaxations (set sizes in this illustration are not to scale).}
\label{fig:sdp_sets}
\end{figure}

\paragraph*{Relations of the Power Flow Relaxations}

From the results presented in Table \ref{tbl:gaps_time}, we can conclude that the QC and lifted nonlinear cuts have different strengths and weaknesses and one does not dominate the other.
Using this information, Figure \ref{fig:sdp_sets} presents an updated Venn Diagram of relaxations (originally presented in \cite{qc_opf_tps}) to reflect the various strengthened relaxations considered here.

\section{Conclusion}
\label{sec:conclusion}

With several years of steady progress on convex relaxations of the AC power flow equations, the optimality gap on the vast majority of AC Optimal Power Flow (AC-OPF) test cases has been closed to less than 1\%.
This paper sought to push the limits of convex relaxations even further and close the optimality gap on the 16 remaining open test cases. 
To that end, the SDP-N+QC+LNC power flow relaxation was developed by hybridizing the SDP and QC relaxations, proposing lifted nonlinear cuts, and performing bounds propagation.
The proposed model was able to reduce the optimality gap to less than 1\% on 8 of the 16 open cases.
Overall, this approach was able to close the gap on 88.7\% of the 71 AC-OPF cases considered herein.

The key weakness of the SDP-N+QC+LNC relaxation is its reliance on SDP solving technology, which suffers from scalability limitations \cite{qc_opf_tps}.
Fortunately, recent works have proposed promising approaches for scaling the SDP relaxations to larger test cases \cite{sdp_cuts_report, strong_soc_report}.
Despite the current scalability challenges, it may still be beneficial to perform this costly SDP computation at the root node of a branch-and-bound method for proving a tight lower bound.  
Indeed, after ten hours of computation, off-the-shelf global optimization solvers \cite{scip,belotti2009couenne} cannot close the optimality gap on the vast majority of AC-OPF test cases.

Thinking more broadly, this work highlights two notable facts about the classic AC-OPF problem.  
First, interior point methods (e.g., Ipopt) are able to find globally optimal solutions in the vast majority of test cases.
Second, it is possible to enclose the non-convex AC-OPF feasibility region in tight convex set, leading to convex relaxations providing very small optimality gaps.
Both of these results are interesting given that the AC-OPF is a non-convex optimization problem, which is known to be NP-Hard in general \cite{verma2009power,ACSTAR2015}.

\bibliographystyle{plain}
\bibliography{../power_models}

\clearpage
\appendix

\section{Analysis of Extreme Values of $W_{ij}$}
\label{sec:w_bounds}

This appendix analyzes the extreme points of $W_{ij}$ in the nonconvex Model \ref{model:ac_opf_w} to develop valid variable bounds for $W_{ij}$.
We begin by developing some basic properties about the minimum and maximum values of various functions.
Then we will compose these properties to develop valid bounds for $W_{ij}$.

Let $f(x,y) = xy$. First we consider the case of multiplying two positive numbers, namely,
\begin{subequations}
\begin{align}
0 \leq \bm {x^l} \leq \bm {x^u}, 0 \leq \bm {y^l} \leq \bm {y^u}
\end{align}
\end{subequations}
observe that the values of $f(x,y)$ are ordered as follows,
\begin{subequations}
\begin{align}
f(\bm {x^l},\bm {y^l}) \leq f(\bm {x^u},\bm {y^l}), f(\bm {x^l},\bm {y^u}) \leq f(\bm {x^u},\bm {y^u})
\end{align}
\end{subequations}
consequently we have,
\begin{lemma} 
\label{lemma:pp_prod}
For positive $x$ and $y$,
\begin{subequations}
\begin{align}
\min(xy) &= \bm {x^l}\bm {y^l} \\
\max(xy) &= \bm {x^u}\bm {y^u}
\end{align}
\end{subequations}
\end{lemma}
\noindent
Second we consider the case of multiplying a positive number with a negative number, namely,
\begin{subequations}
\begin{align}
0 \leq \bm {x^l} \leq \bm {x^u}, \bm {y^l} \leq \bm {y^u} \leq 0
\end{align}
\end{subequations}
observe that the values of $f(x,y)$ are ordered as follows,
\begin{subequations}
\begin{align}
f(\bm {x^u},\bm {y^l}) \leq f(\bm {x^l},\bm {y^l}), f(\bm {x^u},\bm {y^u}) \leq f(\bm {x^l},\bm {y^u})
\end{align}
\end{subequations}
consequently we have,
\begin{lemma} 
\label{lemma:pu_prod} 
For positive $x$ and negative $y$,
\begin{subequations}
\begin{align}
\min(xy) &= \bm {x^u}\bm {y^l} \\
\max(xy) &= \bm {x^l}\bm {y^u}
\end{align}
\end{subequations}
\end{lemma} 
\noindent
Third we consider the case of multiplying a positive number with a negative or positive number, namely,
\begin{subequations}
\begin{align}
0 \leq \bm {x^l} \leq \bm {x^u}, \bm {y^l} \leq 0 \leq \bm {y^u}
\end{align}
\end{subequations}
observe that the values of $f(x,y)$ are ordered as follows,
\begin{subequations}
\begin{align}
f(\bm {x^u},\bm {y^l}) \leq f(\bm {x^l},\bm {y^l}) \leq f(\bm {x^l},\bm {y^u}) \leq f(\bm {x^u},\bm {y^u})
\end{align}
\end{subequations}
consequently we have,
\begin{lemma} 
\label{lemma:pu_prod1} 
For positive $x$ and positive or negative $y$,
\begin{subequations}
\begin{align}
\min(xy) &= \bm {x^u}\bm {y^l} \\
\max(xy) &= \bm {x^u}\bm {y^u}
\end{align}
\end{subequations}
\end{lemma} 

\noindent
Next we consider the extreme values of $f(x) = \sin(x)$ and $g(x) = \cos(x)$ on the interval $-\bm \pi/2 \leq x \leq \bm \pi/2$.
Observing that $\cos(x)$ is non-monotone and has an inflection point at $x=0$, we will break this into three cases based on if the range includes the inflection point, specifically, $-\bm \pi/2 \leq \bm {x^l} \leq \bm {x^u} \leq 0$, $-\bm \pi/2 \leq \bm {x^l} < 0 < \bm {x^u} \leq \bm \pi/2$, and $0 \leq \bm {x^l} \leq \bm {x^u} \leq \bm \pi/2$.
For the first interval $\cos(x)$ is monotone increasing, thus,
\begin{lemma} 
\label{lemma:cos_bounds1} 
For $-\bm \pi/2 \leq \bm {x^l} \leq \bm {x^u} \leq 0$, 
\begin{subequations}
\begin{align}
\min(\cos(x)) &= \cos(\bm {x^l}) \\
\max(\cos(x)) &= \cos(\bm {x^u})
\end{align}
\end{subequations}
\end{lemma}

\noindent
For the second interval $\cos(x)$ passes through the inflection point and this is the maximum value at $x=0$.  The minimum value can occur on either side (i.e. $x<0$ or $x>0$) depending interval, however because both sides are monotone we know the minimum value will occur at one of the extreme points.
\begin{lemma} 
\label{lemma:cos_bounds2} 
For $-\bm \pi/2 \leq \bm {x^l} < 0 < \bm {x^u} \leq \bm \pi/2$, 
\begin{subequations}
\begin{align}
\min(\cos(x)) &= \min(\cos(\bm {x^l}), \cos(\bm {x^u})) \\
\max(\cos(x)) &= \cos(0) = 1
\end{align}
\end{subequations}
\end{lemma}

\noindent
For the third interval $\cos(x)$ is monotone decreasing, thus,
\begin{lemma} 
\label{lemma:cos_bounds3} 
For $0 \leq \bm {x^l} \leq \bm {x^u} \leq \bm \pi/2$, 
\begin{subequations}
\begin{align}
\min(\cos(x)) &= \cos(\bm {x^u}) \\
\max(\cos(x)) &= \cos(\bm {x^l})
\end{align}
\end{subequations}
\end{lemma}

Given that $\sin(x)$ is monotone increasing over the complete range of $x$ only one case is nessiary,
\begin{lemma} 
\label{lemma:sin_bounds} 
For $-\bm \pi/2 \leq x \leq \bm \pi/2$, 
\begin{subequations}
\begin{align}
\min(\sin(x)) &= \sin(\bm {x^l}) \\
\max(\sin(x)) &= \sin(\bm {x^u})
\end{align}
\end{subequations}
\end{lemma}
\noindent
However, it is important to note that $\sin(x)$ is negative for $x < 0$ and positive for $x \geq 0$.
This is the only function considered thus far, which can yield negative values.

With these basic properties defined we are now in a position to develop bounds on $W_{ij}$.
We begin by noting the following real valued interpretation of $W_{ij}$,
\begin{subequations}
\begin{align}
\Re(W_{ij}) = w^R_{ij} = v_i v_j \cos(\theta)\\
\Im(W_{ij}) = w^I_{ij} = v_i v_j \sin(\theta)
\end{align}
\end{subequations}
and the variable bounds from Model \ref{model:ac_opf_w},
\begin{subequations}
\begin{align}
\bm {v^l}_i & \leq v_i \leq \bm {v^u}_i  \\
\bm {v^l}_j & \leq v_j \leq \bm {v^u}_j \\
-\bm \pi/2 \leq \bm {\theta^l}_{ij} & \leq \theta_{ij} \leq \bm {\theta^u}_{ij} \leq \bm \pi/2
\end{align}
\end{subequations}
Next we can compute values for $\min,\max$ of $w^R_{ij} , w^I_{ij}$ by composing the properties developed previously.
The analysis is broken into three cases based on the bounds of $\theta_{ij}$, to account for the inflection point in the cosine function.

\clearpage
\subsection*{Case 1}
In the range $-\bm \pi/2 \leq \bm {x^l} \leq \bm {x^u} \leq 0$,

\begin{table*}[h!]
\center
\begin{tabular}{cccccc}
 $\min(w^R_{ij})$ &   $\max(w^R_{ij})$ & $\min(w^I_{ij})$& $\max(w^I_{ij})$ \\
 $\min(v_i v_j \cos(\theta))$ &  $\max(v_i v_j \cos(\theta))$ & $\min(v_i v_j \sin(\theta))$ & $\max(v_i v_j \sin(\theta))$ \\
 $\min(v_i v_j) \min(\cos(\theta))$ & $\max(v_i v_j) \max(\cos(\theta))$ & $\max(v_i v_j) \min(\sin(\theta))$ &  $\min(v_i v_j) \max( \sin(\theta))$ \\
 $\bm {v^l}_i \bm {v^l}_j \cos(\bm {\theta^l}_{ij})$ & $\bm {v^u}_i \bm {v^u}_j \cos(\bm {\theta^u}_{ij})$ & $\bm {v^u}_i \bm {v^u}_j \sin(\bm {\theta^l}_{ij})$ &  $\bm {v^l}_i \bm {v^l}_j \sin(\bm {\theta^u}_{ij})$ \\
\end{tabular}
\label{tbl:bounds_1}
\end{table*}

\subsection*{Case 2}
In the range $-\bm \pi/2 \leq \bm {x^l} < 0 < \bm {x^u} \leq \bm \pi/2$,

{\center
\begin{tabular}{cccccc}
 $\min(w^R_{ij})$ &   $\max(w^R_{ij})$ & $\min(w^I_{ij})$& $\max(w^I_{ij})$ \\
 $\min(v_i v_j \cos(\theta))$ &  $\max(v_i v_j \cos(\theta))$ & $\min(v_i v_j \sin(\theta))$ & $\max(v_i v_j \sin(\theta))$ \\
 $\min(v_i v_j) \min(\cos(\theta))$ & $\max(v_i v_j) \max(\cos(\theta))$ & $\max(v_i v_j) \min(\sin(\theta))$ &  $\max(v_i v_j) \max( \sin(\theta))$ \\
 $\bm {v^l}_i \bm {v^l}_j \min(\cos(\bm {\theta^l}_{ij}), \cos(\bm {\theta^u}_{ij}))$ & $\bm {v^u}_i \bm {v^u}_j $ & $\bm {v^u}_i \bm {v^u}_j \sin(\bm {\theta^l}_{ij})$ &  $\bm {v^u}_i \bm {v^u}_j \sin(\bm {\theta^u}_{ij})$ \\
\end{tabular}
}

\subsection*{Case 3}
In the range $0 \leq \bm {x^l} \leq \bm {x^u} \leq \bm \pi/2$,

\begin{table*}[h!]
\center
\begin{tabular}{cccccc}
 $\min(w^R_{ij})$ &   $\max(w^R_{ij})$ & $\min(w^I_{ij})$& $\max(w^I_{ij})$ \\
 $\min(v_i v_j \cos(\theta))$ &  $\max(v_i v_j \cos(\theta))$ & $\min(v_i v_j \sin(\theta))$ & $\max(v_i v_j \sin(\theta))$ \\
 $\min(v_i v_j) \min(\cos(\theta))$ & $\max(v_i v_j) \max(\cos(\theta))$ & $\min(v_i v_j) \min(\sin(\theta))$ &  $\max(v_i v_j) \max( \sin(\theta))$ \\
 $\bm {v^l}_i \bm {v^l}_j \cos(\bm {\theta^u}_{ij})$ & $\bm {v^u}_i \bm {v^u}_j \cos(\bm {\theta^l}_{ij})$ & $\bm {v^l}_i \bm {v^l}_j \sin(\bm {\theta^l}_{ij})$ &  $\bm {v^u}_i \bm {v^u}_j \sin(\bm {\theta^u}_{ij})$ \\
\end{tabular}
\label{tbl:bounds_3}
\end{table*}

\noindent
Through these basic properties and utilizing bounds propagation, we have effectively developed valid bounds for $W_{ij}$.

\section{Derivation of Cuts from \cite{7056568}}
\label{sec:w_bound_cut}

In the interest of brevity we only consider the case where, $0 \leq \bm {\theta^l}_{ij} < \bm {\theta^u}_{ij} \leq \bm \pi/2$ and use the standard definition $\bm \phi_{ij} = (\bm {\theta^u}_{ij} + \bm {\theta^l}_{ij})/2, \bm \delta_{ij} = (\bm {\theta^u}_{ij} - \bm {\theta^l}_{ij})/2$. 
Following the algorithm from \cite{7056568}, we first must determine which of four cut cases this situation falls into.
First we compute the values for the bounds on $w^R_{ij}, w^I_{ij}$.
\begin{subequations}
\begin{align}
\bm {w^{Rl}}_{ij} &= \bm {v^l}_i \bm {v^l}_j \cos(\bm {\theta^u}_{ij}) \\
\bm {w^{Ru}}_{ij} &= \bm {v^u}_i \bm {v^u}_j \cos(\bm {\theta^l}_{ij}) \\
\bm {w^{Il}}_{ij}, &= \bm {v^l}_i \bm {v^l}_j \sin(\bm {\theta^l}_{ij}) \\
\bm {w^{Iu}}_{ij} &= \bm {v^u}_i \bm {v^u}_j \sin(\bm {\theta^u}_{ij})
\end{align}
\end{subequations}
Then we evaluate the values of $(\bm {w^{Rl}}_{ij})^2 + (\bm {w^{Il}}_{ij})^2, (\bm {w^{Rl}}_{ij})^2 + (\bm {w^{Iu}}_{ij})^2$ and compare them to $(\bm {v^l}_i \bm {v^l}_j)^2$.
We observe that,
\begin{subequations}
\begin{align}
(\bm {w^{Rl}}_{ij})^2 + (\bm {w^{Il}}_{ij})^2  &= (\bm {v^l}_i \bm {v^l}_j)^2 ( \cos(\bm {\theta^u}_{ij})^2 + \sin(\bm {\theta^l}_{ij})^2) \leq (\bm {v^l}_i \bm {v^l}_j)^2  \\
(\bm {w^{Rl}}_{ij})^2 + (\bm {w^{Iu}}_{ij})^2 &= (\bm {v^l}_i \bm {v^l}_j ( \cos(\bm {\theta^u}_{ij}))^2 + (\bm {v^u}_i \bm {v^u}_j \sin(\bm {\theta^u}_{ij}))^2 \geq (\bm {v^l}_i \bm {v^l}_j)^2 
\end{align}
\end{subequations}
Following the algorithm from \cite{7056568}, this falls into {\em Case 2}.  Next we compute two points,
\begin{subequations}
\begin{align}
\bm x_1 &= \bm {v^l}_i \bm {v^l}_j \cos(\bm {\theta^u}_{ij}) \\
\bm y_1 &= \sqrt{(\bm {v^l}_i \bm {v^l}_j)^2 - \bm {v^l}_i \bm {v^l}_j \cos(\bm {\theta^u}_{ij})} = \bm {v^l}_i \bm {v^l}_j \sin(\bm {\theta^u}_{ij}) \\
\bm x_2 &= \sqrt{(\bm {v^l}_i \bm {v^l}_j)^2 - \bm {v^l}_i \bm {v^l}_j \sin(\bm {\theta^l}_{ij})} = \bm {v^l}_i \bm {v^l}_j \cos(\bm {\theta^l}_{ij}) \\
\bm y_2 &= \bm {v^l}_i \bm {v^l}_j \sin(\bm {\theta^l}_{ij})
\end{align}
\end{subequations}
We now have all the constants required to apply the general cut of \cite{7056568}, which simply fits an inequality between these two points as follows,
\begin{subequations}
\begin{align}
(\bm y_1 - \bm y_2) w^R_{ij} - (\bm x_1 - \bm x_2) w^I_{ij} \geq \bm x_2 \bm y_1 - \bm x_1 \bm y_2
\end{align}
\end{subequations}
expanding in this particular context we have,
\begin{subequations}
\begin{align}
 (\bm {v^l}_i \bm {v^l}_j \sin(\bm {\theta^u}_{ij}) - \bm {v^l}_i \bm {v^l}_j \sin(\bm {\theta^l}_{ij}) ) w^R_{ij} - (\bm {v^l}_i \bm {v^l}_j \cos(\bm {\theta^u}_{ij})  - \bm {v^l}_i \bm {v^l}_j \cos(\bm {\theta^l}_{ij})) w^I_{ij} \nonumber \\ 
 \geq \bm {v^l}_i \bm {v^l}_j \cos(\bm {\theta^l}_{ij}) \bm {v^l}_i \bm {v^l}_j \sin(\bm {\theta^u}_{ij}) - \bm {v^l}_i \bm {v^l}_j \cos(\bm {\theta^u}_{ij})  \bm {v^l}_i \bm {v^l}_j \sin(\bm {\theta^l}_{ij}) \\
 (\sin(\bm {\theta^u}_{ij}) - \sin(\bm {\theta^l}_{ij}) ) w^R_{ij} - (\cos(\bm {\theta^u}_{ij}) - \cos(\bm {\theta^l}_{ij})) w^I_{ij} \geq \bm {v^l}_i \bm {v^l}_j (\cos(\bm {\theta^l}_{ij}) \sin(\bm {\theta^u}_{ij}) -  \sin(\bm {\theta^l}_{ij}) \cos(\bm {\theta^u}_{ij}))  \label{eq:case2_cut_long}
\end{align}
\end{subequations}
To further simplify this formula we make use the following trigonometric identities,
\begin{subequations}
\begin{align}
\sin(x)\cos(y)  &= (\sin(x+y) + \sin(x-y))/2 \nonumber \\
\cos(x)\sin(y)  &= (\sin(x+y) - \sin(x-y))/2 \nonumber \\
\cos(x)\cos(y) &= (\cos(x-y) + \cos(x+y))/2 \nonumber \\
\sin(x)\sin(y)   &= (\cos(x-y) - \cos(x+y))/2 \nonumber \\
\sin(2x)  &= 2 \sin(x)\cos(x) \nonumber
\end{align}
\end{subequations}
And observe the following properties,
\begin{subequations}
\begin{align}
(\cos(\bm {\theta^l}_{ij}) \sin(\bm {\theta^u}_{ij}) -  \sin(\bm {\theta^l}_{ij}) \cos(\bm {\theta^u}_{ij})) \\
-\sin(\bm {\theta^l}_{ij} - \bm {\theta^u}_{ij}) \\
-\sin(\bm \phi_{ij} - \bm \delta_{ij} - \bm \phi_{ij} - \bm \delta_{ij}) \\
\sin(2 \bm \delta_{ij}) \\
2 \sin(\bm \delta_{ij}) \cos(\bm \delta_{ij})
\end{align}
\end{subequations}
\begin{subequations}
\begin{align}
\sin(\bm {\theta^u}_{ij}) - \sin(\bm {\theta^l}_{ij}) \\
\sin(\bm \phi_{ij} + \bm \delta_{ij}) - \sin(\bm \phi_{ij} - \bm \delta_{ij})  \\
2 \cos(\bm \phi_{ij})\sin(\bm \delta_{ij})
\end{align}
\end{subequations}
\begin{subequations}
\begin{align}
\cos(\bm {\theta^l}_{ij}) - \cos(\bm {\theta^u}_{ij})\\
\cos(\bm \phi_{ij} - \bm \delta_{ij}) - \cos(\bm \phi_{ij} + \bm \delta_{ij})  \\
2 \sin(\bm \phi_{ij})\sin(\bm \delta_{ij})
\end{align}
\end{subequations}
Next we apply these properties to \eqref{eq:case2_cut_long} yielding,
\begin{subequations}
\begin{align}
2 \cos(\bm \phi_{ij})\sin(\bm \delta_{ij}) w^R_{ij} + 2 \sin(\bm \phi_{ij})\sin(\bm \delta_{ij}) w^I_{ij} \geq \bm {v^l}_i \bm {v^l}_j (2 \sin(\bm \delta_{ij}) \cos(\bm \delta_{ij})) \\
\cos(\bm \phi_{ij}) w^R_{ij} + \sin(\bm \phi_{ij}) w^I_{ij} \geq \bm {v^l}_i \bm {v^l}_j \cos(\bm \delta_{ij})
\end{align}
\end{subequations}


\section{Derivation of Cuts from \cite{chen_qcqp}}
\label{sec:qcqp_cut}

Utilizing the standard parameterization, $\bm \phi_{ij} = (\bm {\theta^u}_{ij} + \bm {\theta^l}_{ij})/2, \bm \delta_{ij} = (\bm {\theta^u}_{ij} - \bm {\theta^l}_{ij})/2$, throughout this section we will use the following trigonometric identities,
\begin{subequations}
\begin{align}
\sin(\theta^l)\sin(\theta^u) &= \cos(\delta)^2 - \cos(\phi)^2 \nonumber \\
\cos(\theta^l)\cos(\theta^u) &= \cos(\delta)^2 + \cos(\phi)^2 - 1 \nonumber \\
\cos(\theta^l)\sin(\theta^u) + \cos(\theta^u)\sin(\theta^l)  &= 2 \sin(\phi)\cos(\phi) \nonumber \\
\cos(\theta^l) + \cos(\theta^u) &= 2\cos(\delta)\cos(\phi) \nonumber \\
\sin(\theta^l) + \sin(\theta^u) &= 2\cos(\delta)\sin(\phi) \nonumber 
\end{align}
\end{subequations}
The primary challenge in showing the equivalence of \eqref{eq:cc_cut_1}-\eqref{eq:cc_cut_2} to \eqref{eq:4d_cut_1}-\eqref{eq:4d_cut_2}, is the treatment of the following two expressions,
\begin{subequations}
\begin{align}
& f(x,y) = \frac{1 - \left( \frac{\sqrt{1+(x)^2} -1}{x} \right) \left( \frac{\sqrt{1+(y)^2} -1}{y} \right)}{1 + \left( \frac{\sqrt{1+(x)^2} -1}{x} \right) \left( \frac{\sqrt{1+(y)^2} -1}{y} \right)} \label{eq:cc_f} \\
& g(x,y) = \frac{\left( \frac{\sqrt{1+(x)^2} -1}{x} \right) + \left( \frac{\sqrt{1+(y)^2} -1}{y} \right)}{1 + \left( \frac{\sqrt{1+(x)^2} -1}{x} \right) \left( \frac{\sqrt{1+(y)^2} -1}{y} \right)}  \label{eq:cc_g}
\end{align}
\end{subequations}
Hence we will focus our attention on these first.  For brevity, \eqref{eq:cc_f}-\eqref{eq:cc_g} skip the special cases where $x$ or $y$ is 0. 
\begin{lemma}
$f(\tan(\bm {\theta^l}), \tan(\bm {\theta^u})) = \frac{\cos(\bm \phi)}{\cos(\bm \delta)}$
\end{lemma}
\begin{proof}
We begin by developing the denominator of $f(x,y)$ as follows,
\begin{subequations}
\begin{align}
& 1 + \left( \frac{\sqrt{1+\tan(\bm {\theta^l})^2} -1}{\tan(\bm {\theta^l})} \right) \left( \frac{\sqrt{1+\tan(\bm {\theta^u})^2} -1}{\tan(\bm {\theta^u})} \right) \\
& 1 + \left( \frac{\sec(\bm {\theta^l}) -1}{\tan(\bm {\theta^l})} \right) \left( \frac{\sec(\bm {\theta^u}) -1}{\tan(\bm {\theta^u})} \right) \\
& 1 + \left( \frac{1 - \cos(\bm {\theta^l})}{\sin(\bm {\theta^l})} \right) \left( \frac{1 - \cos(\bm {\theta^u})}{\sin(\bm {\theta^u})} \right) \\
& 1 + \frac{1 - \cos(\bm {\theta^l}) - \cos(\bm {\theta^u}) + \cos(\bm {\theta^l}) \cos(\bm {\theta^u}) }{\sin(\bm {\theta^l})\sin(\bm {\theta^u})} \\
& 1 + \frac{1 - 2\cos(\bm {\phi}) \cos(\bm {\delta}) + \cos(\bm {\phi})^2 + \cos(\bm {\delta})^2 - 1 }{\cos(\bm {\delta})^2 - \cos(\bm {\phi})^2} \\
& \frac{2 \cos(\bm {\delta})^2 - 2 \cos(\bm {\phi}) \cos(\bm {\delta}) }{\cos(\bm {\delta})^2 - \cos(\bm {\phi})^2} 
\end{align}
\end{subequations}
Similarly the numerator develops into,
\begin{subequations}
\begin{align}
& \frac{-2 \cos(\bm {\phi})^2 + 2 \cos(\bm {\phi}) \cos(\bm {\delta}) }{\cos(\bm {\delta})^2 - \cos(\bm {\phi})^2} 
\end{align}
\end{subequations}
Combining the numerator and denominator yields,
\begin{subequations}
\begin{align}
& \frac{-\cos(\bm {\phi})^2 + \cos(\bm {\phi}) \cos(\bm {\delta}) }{\cos(\bm {\delta})^2 - \cos(\bm {\phi}) \cos(\bm {\delta}) } \\
& \frac{\cos(\bm {\phi}) (\cos(\bm {\delta}) - \cos(\bm {\phi})) }{\cos(\bm {\delta}) (\cos(\bm {\delta}) - \cos(\bm {\phi}))} \\
& \frac{\cos(\bm {\phi})}{\cos(\bm {\delta})} 
\end{align}
\end{subequations}
completing the proof.
\end{proof}

\begin{lemma}
$g(\tan(\bm {\theta^l}), \tan(\bm {\theta^u})) = \frac{\sin(\bm \phi)}{\cos(\bm \delta)}$
\end{lemma}
\begin{proof}
We begin by developing the numerator of $f(x,y)$ as follows,
\begin{subequations}
\begin{align}
& \left( \frac{\sqrt{1+\tan(\bm {\theta^l})^2} -1}{\tan(\bm {\theta^l})} \right) + \left( \frac{\sqrt{1+\tan(\bm {\theta^u})^2} -1}{\tan(\bm {\theta^u})} \right) \\
& \left( \frac{\sec(\bm {\theta^l}) -1}{\tan(\bm {\theta^l})} \right) + \left( \frac{\sec(\bm {\theta^u}) -1}{\tan(\bm {\theta^u})} \right) \\
& \left(\frac{1- \cos(\bm {\theta^l})}{\sin(\bm {\theta^l})} \right) + \left( \frac{1- \cos(\bm {\theta^u})}{\sin(\bm {\theta^u})} \right) \\
& \frac{ \sin(\bm {\theta^l}) + \sin(\bm {\theta^u}) - \cos(\bm {\theta^l}) \sin(\bm {\theta^u}) - \cos(\bm {\theta^u}) \sin(\bm {\theta^l})}{\sin(\bm {\theta^l})\sin(\bm {\theta^u})} \\
& \frac{ 2 \sin(\bm {\phi}) \cos(\bm {\delta}) - 2  \sin(\bm {\phi}) \cos(\bm {\phi})}{\cos(\bm {\delta})^2 - \cos(\bm {\phi})^2} 
\end{align}
\end{subequations}
Combining the numerator and denominator yields,
\begin{subequations}
\begin{align}
& \frac{\sin(\bm {\phi}) \cos(\bm {\delta}) - \sin(\bm {\phi}) \cos(\bm {\phi}) }{\cos(\bm {\delta})^2 - \cos(\bm {\phi}) \cos(\bm {\delta}) } \\
& \frac{\sin(\bm {\phi}) (\cos(\bm {\delta}) - \cos(\bm {\phi})) }{\cos(\bm {\delta}) (\cos(\bm {\delta}) - \cos(\bm {\phi}))} \\
& \frac{\sin(\bm {\phi})}{\cos(\bm {\delta})} 
\end{align}
\end{subequations}
completing the proof.
\end{proof}

With the simplified arithmetic form of \eqref{eq:cc_f} and \eqref{eq:cc_g}, we can now concisely write the constants from \cite{chen_qcqp} (i.e. \eqref{eq:cc_const_1}-\eqref{eq:cc_const_5} ) using the parameters of this paper as follows,
%
\begin{subequations}
\begin{align}
& \bm \pi_{0} = - \bm {v^l}_{i} \bm {v^l}_{j} \bm {v^u}_{i} \bm {v^u}_{j} \\
& \bm \pi_{1} = - \bm {v^l}_{j} \bm {v^u}_{j} \\
& \bm \pi_{2} = - \bm {v^l}_{i} \bm {v^u}_{i} \\
& \bm \pi_{3} = \bm {v^\sigma}_i \bm {v^\sigma}_j \frac{\cos(\bm \phi_{ij})}{\cos(\bm \delta_{ij})} \\
& \bm \pi_{4} = \bm {v^\sigma}_i \bm {v^\sigma}_j \frac{\sin(\bm \phi_{ij})}{\cos(\bm \delta_{ij})}  
\end{align}
\end{subequations}
Presentation of the cuts proposed in \cite{chen_qcqp} in the format used here is,
\begin{subequations}
\begin{align}
 \bm \pi_{3} w^R_{ij} + \bm \pi_{4} w^I_{ij} + (\bm \pi_{1} - (\bm {v^u}_{j})^2) w_i + (\bm \pi_{2} - (\bm {v^u}_{i})^2) w_j &\geq  - (\bm \pi_{0} + (\bm {v^u}_{i})^2 (\bm {v^u}_{j})^2) \label{eq:cc_cut_1_alt} \\
 \bm \pi_{3} w^R_{ij} + \bm \pi_{4} w^I_{ij} + (\bm \pi_{1} - (\bm {v^l}_{j})^2) w_i + (\bm \pi_{2} - (\bm {v^l}_{i})^2) w_j &\geq  - (\bm \pi_{0} + (\bm {v^l}_{i})^2 (\bm {v^l}_{j})^2) \label{eq:cc_cut_2_alt} 
\end{align}
\end{subequations}

\begin{lemma}
The inequalities \eqref{eq:cc_cut_1_alt}-\eqref{eq:cc_cut_2_alt} from \cite{chen_qcqp} are equivalent to the lifted nonlinear cuts \eqref{eq:4d_cut_1}-\eqref{eq:4d_cut_2}, respectively.
\end{lemma}
\begin{proof}
Expanding the $\bm \pi$ constants and factoring the common terms yields the result.
\end{proof}

\end{document}